\documentclass[10pt,a4paper]{article}

\usepackage[utf8]{inputenc}

\usepackage{a4wide}
\usepackage{fullpage}
\usepackage{latexsym}
\usepackage{amssymb,amsmath,amsthm}
\usepackage{graphicx}
\usepackage{url}
\usepackage{hyperref}
\usepackage{color}
\usepackage{enumerate}
\usepackage[capitalize]{cleveref}
\usepackage{authblk}
\usepackage{tikz-cd}
\usetikzlibrary{nfold}
\usepackage[normalem]{ulem}

\newtheorem{theorem}{Theorem}
\newtheorem*{theorem*}{Theorem}
\newtheorem{proposition}[theorem]{Proposition}

\newtheorem{lemma}[theorem]{Lemma}
\newtheorem{corollary}[theorem]{Corollary}

\newtheorem{claim}{Claim}[theorem]
\theoremstyle{definition}

\theoremstyle{remark}
\newtheorem{remark}[theorem]{Remark}

\makeatletter
\newcommand{\ProofEndBox}{{\ifhmode\unskip\nobreak\hfil\penalty50 \else
          \leavevmode\fi\quad\vadjust{}\nobreak\hfill$\Box$
            \finalhyphendemerits=0 \par}}
\makeatother

\newcommand{\R}{{\mathbb{R}}}

\newcommand{\Q}{{\mathbb{Q}}}

\newcommand{\CP}{{\mathbb{C}P}}

\DeclareMathOperator{\RHS}{RHS}
\DeclareMathOperator{\LHS}{LHS}

\newcommand{\Z}{{\mathbb{Z}}}

\newcommand\BB{\mathcal{B}}
\newcommand\DD{\mathcal{D}}

\newcommand\GG{\mathcal{G}}

\newcommand\PP{\mathcal{P}}
\newcommand\RR{\mathcal{R}}

\newcommand\bP{\mathbf{P}}

\newcommand\makevec[1]{{\bf #1}}

\def \pp {\makevec{p}}

\def \ww {\makevec{w}}

\def \dd {\makevec{d}}
\def \ee {\makevec{e}}

\DeclareMathOperator{\sub}{sub}

\newcommand{\bequal}{$\raisebox{1pt}{$\kern1.5mm
  {\scriptscriptstyle\circ}\kern-2.1mm\rule[-.4mm]{3mm}{.5pt}\kern-3mm\rule[1.2mm]{3mm}{.5pt}\kern1mm$}$}

\newcounter{sideremark}
\newcommand{\marrow}{\stepcounter{sideremark}\marginpar{$\boldsymbol{\longleftarrow\scriptstyle\arabic{sideremark}}$}}

\newcommand{\martin}[1]{{\color{blue}\vskip 5pt\textsf{*** (Martin) \marrow #1\vskip 5pt}}}

\title{Simpler algorithmically unrecognizable 4-manifolds\thanks{Supported by the GA\v{C}R
grant no. 22-19073S.}}
\author[1]{Martin Tancer}

\affil[1]{\small Department of Applied Mathematics, Charles University, Malostransk\'{e} n\'{a}m.
25, 118~00~~Praha~1, Czech Republic}

\date{}

\begin{document}
\maketitle
\begin{abstract}
  Markov proved that there exists an \emph{unrecognizable} 4-manifold, that is,
  a 4-manifold for which the homeomorphism
  problem is undecidable. In this paper we consider the question how close 
  we can get to $S^4$ with an unrecognizable manifold. One of our achievements is
  that we show a way to remove so-called Markov's trick from the proof of
  existence of such a manifold. This trick contributes to the complexity of the
  resulting manifold. We also show how to decrease the deficiency (or the
  number of relations) in so-called Adian-Rabin set which is another ingredient that
  contributes to the complexity of the resulting manifold. Altogether, our
  approach allows to show that the connected sum $\#_9(S^2 \times S^2)$ is
  unrecognizable while the previous best result is the unrecognizability of
  $\#_{12}(S^2 \times S^2)$ due to Gordon.
\end{abstract}

\section{Introduction}

\paragraph{Unrecognizable $4$-manifolds.}
One of the basic questions in topology is whether two topological spaces are
homeomorphic. Even when we restrict our attention to nice spaces such as
manifolds, this problem cannot be (in general) treated algorithmically as
Novikov~\cite{volodin-kuznetsov-fomenko74} has shown that the recognition of
$d$-sphere for $d \geq 5$ is algorithmically unsolvable. In lower
dimensions the homeomorphism problem for manifolds is trivial for $d = 1$ and
easy for $d=2$ (by classification of surfaces). In dimension $3$, keeping the
complexity questions aside, the homeomorphism problem of closed orientable
$3$-manifolds is decidable; see~\cite{kuperberg19}.

In this paper, we focus on $d=4$ for which the decidability questions are not
yet fully understood. On the one hand, it is not known whether recognition of
$4$-sphere is algorithmically solvable. On the other hand, there are known
$4$-manifolds for which the recognition problem is undecidable. From now on, we
say that a manifold $M$ is \emph{unrecognizable}, if it is algorithmically
undecidable whether a given triangulated manifold $N$ is homeomorphic to $M$.
Markov~\cite{markov60} has shown that there exists an unrecognizable
$4$-manifold and it turns out that this manifold is $\#_k(S^2 \times S^2)$, the
connected sum of $k$ copies of $S^2 \times S^2$ for some $k$. 

Because recognizability of $S^4$ is unknown it is also meaningful to ask how
close can we make it to $S^4$ with unrecognizable manifolds; that is, we aim to
find simplest possible unrecognizable manifolds. Note that $\#_k(S^2 \times
S^2)$ is simply connected. If we restrict ourselves to simply connected
manifolds, then we can measure the proximity of $M$ to $S^4$ by the dimension
of the middle homology, that is, by $\beta_2(M)$ where $\beta_i$ denotes the
$i$th Betti number. We remark that $\beta_2(\#_k (S^2 \times S^2)) = 2k$.

When optimizing $k$ above, Stan'ko~\cite{stanko05} has shown that $\#_{14} (S^2
\times S^2)$ (see also~\cite{chernavsky-leksine06}) is unrecognizable. This was
further recently improved by Gordon~\cite{gordon21} to showing that $\#_{12}
(S^2 \times S^2)$ is unrecognizable. We will push this further by showing that
$\#_9 (S^2 \times S^2)$ is unrecognizable; see Corollary~\ref{c:9} below. 
However, apart from the improvement itself, our motivation is also to present
the ideas that yield this improvement. We sketch some of the ideas below the
statement of Corollary~\ref{c:9} in a paragraph `Proof method'.


\paragraph{Group presentations and Markov's method.}
We will be working with finite presentations of groups of a form
$\PP = \langle x_1, \dots, x_n | r_1, \dots, r_m \rangle$. The letters $x_1,
\dots, x_n$ are \emph{generators} (of the free group $\langle x_1, \dots, x_n
| \rangle$ on these generators) while
$r_1, \dots, r_m$ are \emph{relations}. Formally, the relations are words in
alphabet $\{x_1, \dots x_m$, $x_1^{-1}, \dots, x_m^{-1}\}$ (subject to
identifications $x_ix_i^{-1} = 1 = x_i^{-1}x_i$) and then $\PP$ presents
the quotient of the aforementioned free group by the relations normal subgroup
generated by the relations. (Intuitively, adding $r_j$ to the relations means
identifying $r_j$ with the trivial element $1$.) Given a group presentation
$\PP
= \langle x_1, \dots, x_n | r_1, \dots, r_m
\rangle$, the \emph{deficiency} of $\PP$ is the value $m - n$, that is, the
number of relations minus the number of generators.

Now we briefly describe Markov's method
following~\cite{chernavsky-leksine06, kirby20, gordon21}. The first step
is to
obtain an Adian-Rabin set $\bP$ in which each presentation has $m$
relations where $m$ is some fixed number. By
an \emph{Adian-Rabin set} we mean a set of recursively enumerable group presentations such
that there is no algorithm to decide whether a given presentation $\PP \in
\bP$ presents a trivial group.  The next step is to build, for a given
presentation $\PP \in \bP$ presenting a group $G_{\PP}$, a 5-manifold $M_{\PP}$ such that
$\pi_1(\partial M_{\PP}) \cong G_{\PP}$. The final step, called Markov's trick, is to
`stabilize' $\partial M_{\PP}$ by suitably adding to it additional $2$-handles,
obtaining a 4-manifold say $N_{\PP}$. Then, it is possible to show that $N_{\PP} \cong
\#_m(S^2 \times S^2)$ if and only if $G_{\PP}$ is a trivial group.

It follows from a work of Gordon~\cite{gordon95}, building on earlier results
of Borisov~\cite{borisov69} and Matiyasevich~\cite{matiyasevich67}, that there is an Adian-Rabin set with 14 relations which explains that
$\#_{14}(S^2 \times S^2)$ is unrecognizable. In order to get that $\#_{12}(S^2
\times S^2)$ is unrecognizable, Gordon~\cite{gordon21} provided an Adian-Rabin
set with 13 relations as well as he has shown that it is sufficient to add one
less additional handle when `stabilizing'.

\paragraph{Our approach.}
As we wrote earlier, our aim is to remove the Markov trick from the proof. This
in turn means that the complexity of the unrecognizable manifold is not
corresponding to the number of relations but rather to the deficiency of the
presentation.

\begin{theorem}
\label{t:no_markov}
  Assume that there is an Adian-Rabin set where each presentation has deficiency $k$. Then $\#_k(S^2 \times S^2)$ is unrecognizable. 
\end{theorem}

We also provide an Adian-Rabin set with deficiency $9$.

\begin{theorem}
\label{t:def9}
  There exists an Adian-Rabin set where each presentation has deficiency $9$.
\end{theorem}

Then we immediately get:

\begin{corollary}
\label{c:9}
  $\#_9(S^2 \times S^2)$ is unrecognizable.
\end{corollary}

We hope that Theorem~\ref{t:no_markov} may be used for further
improvements getting close to unrecognizability of $S^4$. Minimizing the number
of relations in an Adian-Rabin set, which is a core of the previous approaches,
has been intensively studied in the literature. On the other hand, minimizing
 the deficiency has not been studied. It seems much more realistic to get
 Adian-Rabin sets with very low deficiency, possibly even\footnote{Even
 deficiency 0 could be realistic; however presentations with deficiency 0 are
 intensively studied in the literature under the name balanced presentations.}
1, than Adian-Rabin sets with a very low number of relations.

We also point out that $\#_9(S^2 \times S^2)$ is not the only manifold, for
which we get unrecognizability. Following~\cite{chernavsky-leksine06} we
get: 

\begin{corollary}
\label{c:stabilize}
  Let $M$ be a connected triangulable closed $4$-manifold. Then the connected
  sum $M \# (\#_9(S^2 \times S^2))$ is unrecognizable.
\end{corollary}

For the proof of Corollary~\ref{c:stabilize}, we refer to the proof of
Theorem in Section~5 of~\cite{chernavsky-leksine06}. (Theorem
in~\cite{chernavsky-leksine06} claims an analogous statement with $M \#
(\#_{14}(S^2 \times S^2))$.) The proof partially uses the construction of
$M_{\PP}$ but the only important property is that either $\partial M_{\PP}$ is
homeomorphic to $\#_9(S^2 \times S^2)$ or it has a nontrivial fundamental group.

An example of another simply connected unrecognizable $4$-manifold is the
connected sum $\#_{10}(\CP^2) \# (\#_9(-\CP^2))$ where $\CP^2$ is the complex
projective space and $-\CP^2$ stands for $\CP^2$ with the opposite orientation.
This example is obtained from Corollary~\ref{c:stabilize} by setting $M =
\CP^2$. Note that $(S^2 \times S^2) \# \CP^2 \cong \CP^2 \# \CP^2 \# (-\CP^2)$;
see~\cite[I.~Corollary~4.3]{kirby89}. 

\paragraph{Proof method.}
For the proof of Theorem~\ref{t:no_markov} our main tool is to employ
intersection forms and use theorems of Rohlin and Freedman. The first one
allows to determine the signature of the form; the second one allows to
determine the manifold from the form, if the fundamental group is trivial. At
some point we also need to rule out the connected sum of several copies  $\CP^2
\# (-\CP^2)$. For this we follow ideas of Speyer and Malin. These steps are
significantly different from earlier proofs of Markov's theorem.

The proof of Theorem~\ref{t:def9} is based on modifications of a construction of
Borisov~\cite{borisov69} of a group presentation with 5 letters and 12
relations for which the word problem is undecidable. Such modifications allow
the desired improvement on Adian-Rabin sets but the cost is that we have
to either modify, or completely rework technical proofs of Borisov. In more
detail, we perform two modifications. The first one requires to check that we
can use original Borisov's proof (Theorem~1 in~\cite{borisov69}) essentially in
verbatim (but it is technical to check this). The second modification, however,
requires to rework completely Borisov's embedding theorem (Theorem~2
in~\cite{borisov69}). Here we use so called small cancellation theory which
is essentially based on properties of plane graphs. This part of the
argument is very combinatorial, at some point it even uses properties of
contact representations of disks in the plane, and might be of independent
interest.

\paragraph{Organization of the paper.} Theorem~\ref{t:no_markov} is proved in
Section~\ref{s:no_markov}. For Theorem~\ref{t:def9} we aim to present the main ideas first. From this
reason, we first provide a proof of Theorem~\ref{t:def9} in
Section~\ref{s:modify_borisov} modulo some
auxiliary results. 
In order to
remove doubts of logical dependencies of the results proved, below we present
a diagram of dependencies of the main auxiliary results yielding the proof of
Theorem~\ref{t:def9} (and the sections where these auxiliary results are proved):


\begin{tikzcd}
 \hbox{Thm.~\ref{t:def9} (Sec.~\ref{s:modify_borisov})} &  \arrow[l,Rightarrow] \hbox{Thm.~\ref{t:R_word_unsolvable}  (Sec.~\ref{s:new_encoding})}
  &\arrow[l,Rightarrow] \hbox{Thm.~\ref{t:B'}
  (Appendix~\ref{a:borisov_modified})} \\
   \arrow[u,Rightarrow]
  \hbox{Thm.~\ref{t:ar_sequence} (Sec.~\ref{s:ar_sequence})}    
   & \arrow[u,Rightarrow]\hbox{Thm.~\ref{t:encoding} (Sec.~\ref{s:new_encoding})}
\end{tikzcd}
\paragraph{Terminology used.} We will need a few notions from algebraic
topology throughout the paper which we do not define here, including the
notions of manifold, smooth manifold, simplicial complex, CW-complex or Euler characteristic. In
general, we refer for these notions to textbooks such as~\cite{hatcher02,
tomdieck08}. In Section~\ref{s:new_encoding} we also use some basic notions from
graph theory (graphs, paths, walks, cycles); here we refer to~\cite{diestel18}.

\section{Unrecognizable 4-manifolds without Markov's trick}
\label{s:no_markov}

The aim of this section is to prove Theorem~\ref{t:no_markov}.
For the proof,
we will need several preliminaries. First, we recall some properties of a well
known construction of a manifold $M_{\PP}$ mentioned in the introduction; see,
e.g.,~\cite{boone-haken-poenaru66}.

\begin{proposition}
\label{p:M}
Let $\PP = \langle x_1, \dots, x_n| r_1, \dots, r_m \rangle$ be a 
  presentation of a group $G_{\PP}$. Then there is an algorithmically
  constructible compact $5$-manifold $M_{\PP}$ with
  boundary, possessing the following properties:
  \begin{enumerate}[(i)]
    \item $M_{\PP}$ is smoothly embeddable in $\R^5$;
    \item $\pi_1(M_{\PP}) \cong  \pi_1(\partial  M_{\PP}) \cong G_{\PP}$; and
    \item if $G_{\PP}$ is trivial, then $H_2(\partial M)$ is a free product of $2k$
      copies of $\Z$ where $k$ is the deficiency of $\PP$.
\end{enumerate}
\end{proposition}

\begin{proof}
  The first step is to build a simplicial $2$-complex $K_{\PP}$ such that
  $\pi_1(M_{\PP}) \cong G_{\PP}$. This simplicial complex can be built by
  subdividing
  a CW-complex obtained from a wedge of $n$ 1-spheres (each 1-sphere is
  oriented and corresponds to one of the generators) and attaching $2$-cells to
  the 1-spheres according to the relations in $\PP$. 

  The complex $K_{\PP}$ linearly embeds in $\R^5$ by a general position argument.
  Now we pick three sufficiently small numbers $\rho_0 > \rho_1 > \rho_2 > 0$
  and we consider a manifold $M'_{\PP} \subseteq \R^5$ consisting of all points in
  $\R^5$ in a distance at most $\rho_i$ from some $i$-face of $K_{\PP}$. If the
  radii $\rho_i$ are suitably chosen, then $M'_{\PP}$ deformation retracts to
  $K_{\PP}$. Thus $\pi_1(M'_{\PP}) \cong \pi_1(K_{\PP}) \cong G_{\PP}$. By `smoothening the
  corners' we obtain the desired smooth manifold $M_{\PP}$ smoothly embedded in
  $\R^5$ with $\pi_1(M_{\PP}) \cong \pi_1(M'_{\PP}) \cong G_{\PP}$ ($M_{\PP}$ is homotopy
  equivalent with $M'_{\PP}$). But it also follows
  that $\pi_1(M_{\PP}) \cong \pi_1(\partial M_{\PP})$; see the proof of Lemma~3
  in~\cite{kirby20}. This establishes (i) and (ii).

  For (iii), if $G_\PP$ is trivial, we first note that $H_2(\partial M)$ (and
  $H^2(\partial M)$) are free $\Z$-modules (there is no torsion). Indeed, this
  follows from the Poincaré duality and the universal coefficient theorem
  (using that $H_1(\partial M) = 0$). Thus it is sufficient to perform
  computations over $\Q$.
  Next we 
  note that the Euler characteristic of $K_{\PP}$ equals $1 - n + m =
  1 + k$. (Computing the Euler characteristic is easy. This formula is, for
  example, mentioned below proof of Theorem~2.7 in Chapter I.
  of~\cite{hog-angeloni-metzler-sieradski93}.)  The Euler characteristic can be
  equivalently expressed as $\beta_0(K_{\PP}) - \beta_1(K_{\PP}) +
  \beta_2(K_{\PP})$. If
  $G_{\PP}$ is trivial, then $K_{\PP}$ is simply connected and thus
  $\beta_1(K_{\PP}) = 0$.
  Because $K_{\PP}$ is connected, we also get $\beta_0(K_{\PP}) = 1$. Thus we deduce
  that $\beta_2(K_{\PP}) = k$. Consequently, $\beta_2(M_{\PP}) = k$ as
  $M_{\PP}$ and $K_{\PP}$
  are homotopy equivalent. Finally, by `half lives, half dies';
  see~\cite{putman22} or \cite[Lemma 3.5]{hatcher07} (in the latter case, the
  proof is given only in dimension 3 but it extends to odd dimensions), we
  get $\beta_2(\partial M_{\PP}) = 2k$. 
\end{proof}

\paragraph{Intersection form.} Given a closed smooth simply
connected\footnote{The form can be also defined
without the assumptions smooth, closed or simply connected. If we do not assume
simply connected, then it is necessary to factor out the torsion from $H_2(M,
\Z)$. If we do not assume closed, then the form can be defined in (at least
two) different ways yielding different results. In any case, in this paper we
use the form only for closed smooth simply 
connected 4-manifolds.} 4-manifold $M$,
there is a symmetric bilinear form $\Omega \colon H_2(M, \Z) \times H_2(M, \Z) \to \Z$
called the \emph{intersection form}; $H_2(M, \Z)$ is in this case a free
$\Z$-module. For a proper definition and more details
we refer to~\cite{kirby89}. We will only use this form implicitly, using known properties
of the form and some results employing the form. One property we need is
that the form is \emph{unimodular}; that is, the matrix corresponding to the
form has determinant $\pm 1$. 

We will also need a partial characterization of unimodular symmetric bilinear
forms on a free (finitely generated) $\Z$-module $X$; that is, $X$ is
isomorphic to a direct sum of $\ell$ copies of $\Z$. We reproduce the
characterization from~\cite{kirby89}.
From now on, by a form, we
mean an unimodular symmetric bilinear form on $X$. (Unimodularity implies that a
from has a full rank $\ell$.) For a form $\Gamma \colon X \times X \to \Z$, we
say that $\Gamma$ is \emph{positive} (or \emph{negative}) \emph{definite} if $\Gamma(x,x) > 0$ (or
$\Gamma(x,x) < 0$) for every $x \in X, x \neq 0$. Otherwise, $\Gamma$ is
\emph{indefinite}.
The symmetric bilinear forms on $X$ cannot be perhaps fully characterized in
general but they can be fully characterized in the indefinite case. (In this
case, the matrix corresponding to the form has both positive and negative
eigenvalues.) 

A form $\Gamma$ is \emph{even} if $\Gamma(x,x)$ is even for every $x \in X$;
otherwise, it is \emph{odd}. The \emph{signature} of a form is the number of
positive eigenvalues minus the number of negative eigenvalues of a matrix
corresponding to the form. The indefinite forms are fully characterized by
rank, signature and the property whether they are odd or even:

All odd indefinite forms are forms with matrix (in a suitable basis):
\[
  \begin{pmatrix}
    1 & & & & & \\
    & \ddots & & & & \\
    & & 1 & & & \\
    & & & - 1 & & \\
    & & & & \ddots & \\
    & & & & & -1\\
\end{pmatrix}
\]
with $p > 1$ ones and $q > 1$ minus ones. The rank of such form is $\ell = p+q$
and the signature is equal to $p-q$.

All even indefinite forms are those that can be written as (up to change of basis)
$\pm \oplus_r E_8 \oplus_s \begin{pmatrix} 0 & 1 \\ 1 & 0 \\ \end{pmatrix}$
  where
\begin{itemize}
  \item $E_8$ is a specific (positively definite) form of rank $8$.
   \item $\oplus$ stands for a direct sum of forms. If $\Gamma$ is a form with
     matrix $A$ and $\Lambda$ is a form with matrix $B$, then $\Gamma \oplus
     \Lambda$ is the form with block-diagonal matrix $\begin{pmatrix} A & 0 \\
     0 & B \\ \end{pmatrix}$.
   \item With slight abuse of notation, we identify $\begin{pmatrix} 0 & 1 \\ 1
   & 0 \\ \end{pmatrix}$ with the form with this matrix.
 \item The number $s$ of summands $\begin{pmatrix} 0 & 1 \\ 1
   & 0 \\ \end{pmatrix}$ is positive.
\end{itemize}
The rank of such a form is equal to $8r + 2s$, the signature is equal to $\pm
8r$. 


\paragraph{Theorems of Rohlin and Freedman.}

We will also need the following theorems of Rohlin and Freedman.

\begin{theorem}[Rohlin; see, e.g.~{\cite[Chapter~3]{scorpan05}}]
\label{t:rohlin}
  An oriented closed smooth $4$-manifold bounds an oriented $5$-manifold if and only
  if the signature (of the intersection form) is zero. 
\end{theorem}

 We remark that we only need the easier implication: if the $4$-manifold bounds
 a $5$-manifold, then the signature is zero.

The following theorem is an immediate corollary of Theorem~1.5 (uniqueness)
in~\cite{freedman82}. (If the form is odd, only one case in the statement of Theorem~1.5
in~\cite{freedman82} can be smooth.)

\begin{theorem}[Freedman]
\label{t:freedman}
 Given an (integral) unimodular symmetric bilinear form $\Gamma$, then there is at most
  one smooth simply connected 4-manifold (up to a homeomorphism) which realizes $\Gamma$ as its
  intersection form.
\end{theorem}

\paragraph{Non-embeddability of $\#_k (\CP^2 \# - \CP^2)$.}

Last ingredient we need for a proof of Theorem~\ref{t:no_markov} is the
following nonembeddability result. 

\begin{theorem}
\label{t:non_embed}
  Let $N$ be a smooth manifold homeomorphic to $\#_k (\CP^2 \# - \CP^2)$ (but
  not necessarily diffeomorphic). Then $N$ does not admit a smooth embedding
  into $\R^5$.
\end{theorem}

We point out that the special case of nonembeddability of $N$ with the standard
smooth structure has been proved by Speyer in a MathOverflow 
answer~\cite{speyer23}. The proof we present here is strongly inspired by
Speyer's proof as well as by a remark of Malin in the same post.

 In the proof of Theorem~\ref{t:non_embed} we need real vector bundles. We will
 use terminology from Chapter~1 of~\cite{hatcher17}. Intuitively an
 $n$-dimensional vector bundle is a map $p\colon E \to B$ where for every $b \in
 B$ the preimage $p^{-1}(b)$ (called a \emph{fiber}) has a structure of $\R^n$; it depends continuously
 on $b$; and two fibers over distinct points of $B$ are disjoint.
 For a precise definition we refer to~\cite{hatcher17}. The space $E$ is called
 the \emph{total space} and $B$ is called the \emph{base space}. Two real vector
 bundles with the same base space admit a direct sum denoted $\oplus$;
 intuitively, this makes the direct sum in every fiber but for a precise
 definition, we again refer to~\cite{hatcher17}. A real vector bundle $E \to B$ is \emph{stably trivial}
 if there is a positive integer $n$ such that $E \oplus \underline{\R^n}$ is
 trivial where $\underline{\R^n}$ is the trivial bundle over ${\R^n}$.
 Otherwise, $E \to B$ is \emph{stably nontrivial}.

 Every smooth manifold $M$ admits an important bundle $T_M$ called the
 \emph{tangent bundle} of $M$. If $M$ is embedded in $\R^m$ for some $m$, 
 then the fiber at a point $x \in M$ can be identified with a collection of tangent
 vectors at $x$ in the given embedding. We also need Stiefel-Whitney classes
 in the proof but we use them only implicitly in references to the literature.
 
\begin{proof}
The intersection form of $\#_k (\CP^2 \# - \CP^2)$ is well known---up to a choice
of basis it is a form with diagonal matrix with $k$-times $1$ and $k$-times
  $-1$; see~\cite{kirby89,scorpan05}. In particular, this form is odd. As a corollary of Wu's formula,
  explicitly stated in~\cite[Lemma~4.1, Chap. II]{kirby89} or~\cite[Sec.~4.3]{scorpan05}, we get that the second
  Stiefel-Whitney class of the tangent bundle of $N$, that is $w_2(T_N)$,
  is nontrivial. It follows that $T_N$ is
  stably nontrivial; see Chapter~3, including Example~3.5, in~\cite{hatcher17}. 
  
  On the other hand, let us assume that $N$ admits a smooth embedding into
  $\R^5$. Then the tangent bundle of $T_N$ is stably trivial as taking the
  direct sum with the normal bundle (from the embedding) yields a trivial
  bundle (coming from the trivial bundle on $\R^5$ restricted to the embedding
  of $N$; see Section~1.1, `Direct Sums' in~\cite{hatcher17} for an analogous
  example with spheres). A contradiction.
\end{proof}

%

\begin{proof}[Proof of Theorem~\ref{t:no_markov}]
  Let $\PP$ a presentation of deficiency $k$ coming from the Adian-Rabin set in
  the statement. Let $M_{\PP}$ be the manifold from Proposition~\ref{p:M}. We aim to
  show that $\partial M_{\PP} \cong \#_k(S^2 \times S^2)$ if and only if
  $G_{\PP}$
  (presented by $\PP$) is trivial. This will prove the result.

  If $G_{\PP}$ is not trivial, then $\pi_1(\partial M_{\PP})$ is not trivial as well, by
  Proposition~\ref{p:M}(ii), which means that $M \not\cong \#_k(S^2 \times S^2)$.

  On the other hand, if $G_{\PP}$ is trivial, then we also 
  know that $H_2(\partial M)$ is a free product of $2k$ copies of
  $\Z$ by Proposition~\ref{p:M}(iii). 
  As $\partial M$ bounds $M$, we deduce from Theorem~\ref{t:rohlin} that the
  signature of the intersection form $\Omega$ of $\partial M$ is zero. A
  fortiori, $\Omega$ must be indefinite. By characterization of indefinite
  forms, this leaves two options. The first one is that $\Omega$ is odd and
  $p=q = k$ (using $p$ and $q$ as in characterization earlier). The second
  option is that $\Omega$ is even and $r=0$, $s = k$ (again using $r$ and $s$
  as in characterization earlier). The odd case can be realized by $\#_k(\CP^2
  \# -\CP^2)$ (see~\cite{kirby89}) while the even case can be realized by $\#_k(S^2
  \times S^2)$ (see again~\cite{kirby89}). By Theorem~\ref{t:freedman} we deduce
  that $\partial M$ is homeomorphic either to $\#_k(\CP^2
  \# -\CP^2)$ or $\#_k(S^2 \times S^2)$. However, the case $\#_k(\CP^2
  \# -\CP^2)$ can be excluded due to Proposition~\ref{p:M}(i) and
  Theorem~\ref{t:non_embed}.
\end{proof}

\section{Construction of an Adian-Rabin set}
\label{s:ar_sequence}

The aim of this section and next one is to provide a proof of
Theorem~\ref{t:def9} (modulo some postponed technical results). We will
simultaneously explain an earlier approach how to get an Adian-Rabin set with
bounded number of relations as well as we will introduce new tricks for getting
deficiency $9$. 

\paragraph{Word problem for semigroups.}
The starting point is a result of Matiyasevich~\cite{matiyasevich67} who
constructed a semigroup with $2$ letters and $3$ relations for which the word
problem is undecidable. Equivalently, Matiyasevich constructed so called Thue
system on two letters and with three rewriting rules for which there is a word
$\PP$ and it is undecidable whether a given word $Q$ is equivalent to $\PP$ in this
Thue system. We will briefly describe this Thue system, which will be important
for our next considerations. (Unfortunately, we do not have an access to the
original Matiyasevich~\cite{matiyasevich67} paper (in Russian) but this Thue
system is briefly described in~\cite[Eq. (6)]{matiyasevich95}.)
On alphabet $\{s_1, s_2\}$ the three rewriting rules are the following.
\begin{align*}
  s_1s_1s_2s_1s_2 &\longleftrightarrow s_2s_1s_1 \\
  s_1s_1s_2s_2 &\longleftrightarrow s_2s_1s_1 \\
   L &\longleftrightarrow M\\
\end{align*}
Both $L$ and $M$ above are specific words on alphabet $\{s_1, s_2\}$.
We do not need to know them explicitly, we only point out that $L$ has 304
letters while $M$ has 608 letters. (Two words in a Thue system are equivalent
if one can be obtained from another by a finite number of applications of
rewriting rules applied on subwords formed by consecutive letters.)

\paragraph{Word problem for groups.}
Next step is a construction by Borisov~\cite{borisov69} who provided a group presentation with 5
generators and 12 relations for which the word problem is unsolvable. That is,
there is a word $W$ and it is undecidable whether a given word $U$ is equal to
$W$ in this presentation. However, for groups, this is already equivalent to
asking whether $WU^{-1}$ is trivial. Therefore, it is undecidable whether a
given word is trivial in this presentation. Borisov's presentation is based on
the Thue system above. We will need to describe Borisov's
presentation quite in detail. However, it is more convenient to postpone the
details while explaining the next step first. Then it will be more transparent
what are the desired modifications of Borisov's presentation in order to
achieve an improvement.

\paragraph{From undecidable word problem towards an Adian-Rabin set.} Now let us assume that
$\PP = \langle x_1, \dots, x_n|r_1, \dots, r_m\rangle$ is a group presentation
for which the word problem is undecidable. Gordon~\cite{gordon95} provided a
method how to build an Adian-Rabin set with $m+2$ relators and two
generators. This approach has been also slightly modified by
Miller~\cite{miller92}. We will be using Miller's approach because we find it
easier to access; however, Gordon's earlier\footnote{Gordon's approach is
earlier because it is based on an unpublished manuscript from 1980.} approach
could be used as well. In a recent work~\cite{gordon21}, for a specific case of
Borisov's (set of) presentations, after a slight modification, Gordon shows how the number of
relations can be dropped from 14 to 13. We will formulate this more generally
using one more parameter of a presentation.

\begin{theorem}
\label{t:ar_sequence}
Assume that $\PP = \langle x_1, \dots, x_n|r_1, \dots, r_m\rangle$ is a group
  presentation for which the word problem is undecidable. Assume that we can
  find a set of other relations $\{r'_1, \dots, r'_{m'}\}$ such that
  $\langle x_1, \dots, x_n|r_1, \dots, r_m, r'_1, \dots, r'_{m'}\rangle$ is
  trivial. Then there is an Adian-Rabin set where each presentation has
  deficiency $m + m' - n$.
\end{theorem}

\begin{remark}
  In the statement of Theorem~\ref{t:ar_sequence} we care only about the
  deficiency. From the proof we will easily see that the number of
  generators in each presentation is $n+3$ and the number of relations in each
  presentation is $m + m' + 3$. However, this can be even dropped to $2$ generators and
  $m + m' + 2 - n$ relations; see~\cite{miller92,chernavsky-leksine06}.
\end{remark}

Before we start a proof, we just briefly describe Miller's
approach~\cite{miller92} for comparison: For each word $w$ over $x_1, \dots, x_n$ and their
inverses, consider the presentation where the generators are $x_1,
\dots, x_n$ and three additional letters $\alpha, \beta, \gamma$. The relations
are the original relations in $\PP$ and in addition the following relations
(the last one appears for every generator $x_i$):

\begin{align*}
  \alpha^{-1} \beta \alpha &= \gamma^{-1}\beta^{-1}\gamma\beta\gamma \\
  \alpha^{-2}\beta^{-1}\alpha\beta\alpha^{2} &= 
  \gamma^{-2}\beta^{-1}\gamma\beta\gamma^{2} \\
  \alpha^{-3}[w,\beta]\alpha^{3} &= \gamma^{-3}\beta\gamma^3 \\
  \alpha^{-3+i}x_i\beta\alpha^{3+i} &= \gamma^{-(3+i)}\beta\gamma^{3+i} \\
\end{align*}
where $[w,\beta]$ stands for the commutator $w^{-1}\beta^{-1}w\beta$.

It is easy to check that if $w$ is trivial in $\PP$, then the aforementioned
presentation is trivial. It is not much harder to check that if $w$ is
nontrivial in $\PP$, then actually the aforementioned presentation is
nontrivial. (We will essentially see both steps in the proof of
Theorem~\ref{t:ar_sequence}.) Then the proof contains one more trick how to
express everything in two generators but we do not need this trick (though it
would also apply in our setting).

Now we provide a full proof of Theorem~\ref{t:ar_sequence}. The proof very
closely follows~\cite{miller92} or~\cite{gordon95}.

\begin{proof}[Proof of Theorem~\ref{t:ar_sequence}]
We order the words with the letters $x_1, \dots, x_n$ and their inverses into a
sequence. For each such word we create a presentation $\PP_w$ with the required
deficiency such that $\PP_w$ is trivial if and only if $w$ is trivial. This way
we get the required Adian-Rabin set. 

The generators of $\PP_w$ are $x_1, \dots, x_n, \alpha, \beta, \gamma$. The
  relations are
\begin{align}
  \alpha^{-1} \beta \alpha &= \gamma^{-1}\beta^{-1}\gamma\beta\gamma
  \label{eq:miller_1}\\
  \alpha^{-2}\beta^{-1}\alpha\beta\alpha^{2} &= 
  \gamma^{-2}\beta^{-1}\gamma\beta\gamma^{2} \\
  \alpha^{-3}w\alpha^{3} &= \gamma^{-3}\beta\gamma^3 \\
  \alpha^{-3+i}r'_i\beta\alpha^{3+i} &= \gamma^{-(3+i)}\beta\gamma^{3+i}
  \label{eq:miller_4}
\end{align}
  and the relations $r_1, \dots, r_m$ of $\PP$. (Recall that $r'_i$ are
  the additional relations from the statement of the theorem.)
  
  The change (when compared with Miller) in the third relation is only cosmetic. The changes in the last
  relations are important (and we have one relation for every $r'_i$). Each
  $\PP_w$ has $n+3$ generators and $m + m' + 3$ relations, thereby it has the
  required deficiency.  

First, we show that if $w$ is trivial, then $\PP_w$ presents the trivial group.
Indeed, if $w$ is trivial, then we deduce that $\beta$ is trivial from the
  third relation. Consequently, we deduce that $\gamma$ is trivial from the
  first relation and that $\alpha$ is trivial from the second. Then we deduce
  that each $r'_i$ is trivial from the last set of relations. This means that
  $\PP_w$ is trivial due to our assumption. 

Now we assume that $w$ is nontrivial and our aim is to show that $\PP_w$
  presents a nontrivial group. First we consider a free group $G^{\RHS}$
  presented as $\langle \beta,
  \gamma| \rangle$ with generators $\beta$ and $\gamma$. Then
  $\beta$ together with the right hand sides
  of~\eqref{eq:miller_1}--\eqref{eq:miller_4} form free generators of a
  subgroup of $G^{\RHS}$. Indeed, an arbitrary product of powers of these elements or
  inverses is nontrivial in $G^{\RHS}$ as the middle portions of the word on
  the right hand side of~\eqref{eq:miller_1}--\eqref{eq:miller_4} will
  persist.\footnote{Of course, we consider only nontrivial products in reduced
  forms. That is, no $g^{i}$ and $g^{-j}$ appear next to each other if $i, j
  \in \{1, 2, \dots \}$ and $g$ is either $\beta$ or a word on the right hand
  side of~\eqref{eq:miller_1}--\eqref{eq:miller_4}.} Similarly, we consider a
  group $G^{\LHS}$ which is the free product $\PP * \langle \alpha, \beta
  \rangle$ (identifying a group presentation with the group it presents). Then
  $\beta$ together with the left hand sides
  of~\eqref{eq:miller_1}--\eqref{eq:miller_4} form free generators of a
  subgroup of $G^{\LHS}$ by the same argument as above, assuming that $w$ is
  nontrivial. This means that $\PP_w$ is a presentation of the free product
  $G^{\LHS} * G^{\RHS}$ with amalgamation where the amalgamation is given by
  identifying the symbols $\beta$ and using the
  relations~\eqref{eq:miller_1}--\eqref{eq:miller_4}. This means that
  $G^{\LHS}$ and consequently $\PP$ is a subgroup of $\PP_w$ (again identifying
  groups and presentations). Therefore $\PP_w$ is nontrivial.
\end{proof}

\section{Modifications in Borisov's approach}
\label{s:modify_borisov}

Now we need to describe Borisov's presentation quite in
detail. 
In order to get a
presentation with 5 generators and 12 relations with undecidable word problem,
Borisov proceeds in two steps. First, he provides a presentation with 7
generators and 14 relations. Then by a trick, he shows how to reduce this to
the desired numbers. Our target is to provide such modifications to Borisov's
approach so that we do not affect the number of relations but we make the value
$m' - n$ in Theorem~\ref{t:ar_sequence} as small as possible which yields an
Adian-Rabin set with small deficiency.

\paragraph{The first step in Borisov's approach:} 
The first step of Borisov's approach starts with a Thue system for which the word problem is undecidable. (Of
course, we intend to use the Thue system of Matiyasevich described above. Let
us, however, proceed more generally for the moment.) Because we will at some
point closely follow Borisov's proof pointing out necessary changes, we try to
keep the notation as close as Borisov's original
notation.\footnote{Unfortunately, the are certain inconsistencies
regarding using capital letters, Greek letters, etc. in this notation.}

Let $\{s_1, \dots, s_N\}$
be the alphabet of our Thue system and let $F_i \longleftrightarrow E_i$ be the
rewriting rules for $i \in [M]$. (Here $[M]$ denotes the integer interval $\{1,
\dots, M\}$.) Let $P$ be the word of the Thue system, for which it is
undecidable whether a given word $Q$ is equivalent with $P$. 

Borisov's presentation will be on alphabet $c, d, e, k, s_1, \dots, s_N, t$. 
We will also use auxiliary integers $\alpha$ larger than $M$. 

Now the relations in Borisov's presentation, which we denote $\BB$, are the following. 
\begin{align}
  d^{\alpha} s_\beta &= s_\beta d & &\beta \in [N] \tag{B1} \label{e:b1} \\
  es_{\beta} &= s_{\beta}e^{\alpha} & &\beta \in [N] \tag{B2} \label{e:b2} \\
  s_{\beta}c &= cs_{\beta} & &\beta \in [N] \tag{B3} \label{e:b3} \\
  d^iF_ie^ic &= cd^iE_ie^i & &i \in [M] \tag{B4} \label{e:b4} \\
  ct &= tc & \tag{B5} \label{e:b5} \\
  dt &= td & \tag{B6} \label{e:b6} \\
  ck &= kc & \tag{B7} \label{e:b7} \\
  ek &= ke & \tag{B8} \label{e:b8} \\
  P^{-1}tPk &= kP^{-1}tP &
  \label{e:b9}
  \tag{B9}
\end{align}

Borisov~\cite[Thm.~1]{borisov69} proved the following:

\begin{theorem}[Borisov]
\label{t:borisov}
  Let $Q$ be a word on the alphabet $\{s_1, \dots, s_N\}$ (inverses are not
  allowed). Then $Q$ is equivalent to $P$ in the Thue system underlying this
  presentation if and only if $Q^{-1}tQk =
  kQ^{-1}tQ$ in $\BB$.
\end{theorem}

Using Theorem~\ref{t:borisov}, we see that it is undecidable whether
$Q^{-1}tQk(kQ^{-1}tQ)^{-1}$ is a
trivial word in $\BB$.

\paragraph{Our modification:} We will modify the presentation $\BB$ to a
presentation $\BB'$ in the following way. We consider four injective 
functions $g,h,g',h' \colon [M] \to \{1, 2, \dots \}$. We will assume that
$\alpha$ is larger than any value attained by any of these functions. 
We
replace the relation~\eqref{e:b4} with the following relations.

\begin{align}
  d^{g(i)}F_ie^{h(i)}c &= cd^{g'(i)}E_ie^{h'(i)} & &i \in [M] \tag{B4'} \label{e:b4'}
\end{align}

Other relations are unchanged. The underlying motivation for this modification
is that it will be able to get smaller value of $m'$ in
Theorem~\ref{t:ar_sequence} for this presentation. This improves the
deficiency. We claim that Theorem~\ref{t:borisov} persists with the modification we
made.

\begin{theorem}
\label{t:B'}
  Let $Q$ be a word on the alphabet $\{s_1, \dots, s_N\}$ (inverses are not
  allowed). Then $Q$ is equivalent to $P$ in the Thue system underlying this
  presentation if and only if $Q^{-1}tQk =
  kQ^{-1}tQ$ in $\BB'$.
\end{theorem}

Here we prove the `only if' part of the statement which is the easier part. The
`if' part can be proved in exactly same way as in Borisov's
paper~\cite{borisov69} replacing each occurrence of $d^i F_i e^i$ with
$d^{g(i)} F_i e^{h(i)}$ and of $d^i E_i e^i$ with $d^{g'(i)} E_i e^{h'(i)}$.
In appendix~\ref{a:borisov_modified} we explain in detail the changes in
Borisov's proof so that it works in our setting as well.

\begin{proof}[Proof of `only if' part]
%
%
%
  
  Let us assume that $P$ and $Q$ are equivalent in the Thue system. We want to
  show that $Q^{-1}tQk = kQ^{-1}tQ$ in $\BB'$.
  It is sufficient to prove it in
  the case when $Q$ is obtained from $P$ by a single application of a rewriting
  rule in the Thue system and then to proceed by induction. Let us assume that
  $P = RF_iS$ and $Q = RE_iS$ for some words $R, S$ on the alphabet $\{s_1,
  \dots, s_N\}$ (without inverses) and $i \in [M]$. (The other case $P = RE_iS$
  and $Q = RF_iS$ is analogous.)

  Let $\dd := d^{g(i)\alpha^{|R|}}$, $\dd' := d^{g'(i)\alpha^{|R|}}$, $\ee :=
  e^{h(i)\alpha^{|S|}}$ and $\ee' := e^{h'(i)\alpha^{|S|}}$ where $|W|$ stands for
  the length of a word $W$. First we observe that the following equality holds
  in $\BB'$:

  \begin{equation}
    \label{e:dPec}
    c\dd' Q \ee' = \dd P \ee c
  \end{equation}
Indeed
 \begin{align*}
   c\dd' Q \ee' &= c\dd' RE_iS \ee' \stackrel{(\ref{e:b1},\ref{e:b2})}{=}
   cRd^{g'(i)} E_i  e^{h'(i)} S  \stackrel{\eqref{e:b3}}{=} Rcd^{g'(i)} E_i
   e^{h'(i)} S
   \stackrel{\eqref{e:b4'}}{=} Rd^{g(i)} F_i e^{h(i)} cS \\
   &\stackrel{\eqref{e:b3}}{=} Rd^{g(i)} F_i e^{h(i)} Sc 
   \stackrel{(\ref{e:b1},\ref{e:b2})}{=} \dd RF_iS \ee c = \dd P \ee c. 
 \end{align*}
Now,

  \begin{tabular}{rcl}
    $Q^{-1}tQkQ^{-1}t^{-1}Qk^{-1}$
    &
    $\stackrel{(\ref{e:b5},\ref{e:b6},\ref{e:b8})}{=}$ &
    $\ee' \ee'^{-1} Q^{-1}   \dd'^{-1} c^{-1} t c
    \dd' Q\ee' k \ee'^{-1} Q^{-1} \dd'^{-1} c^{-1} t^{-1} c \dd' Q \ee' k^{-1}
    \ee'^{-1}$ \\
    &
    $\stackrel{\eqref{e:dPec}}{=}$ & $\ee' c^{-1} \ee^{-1} P^{-1} \dd^{-1} t \dd P
    \ee c k c^{-1} \ee^{-1} P^{-1} \dd^{-1} t^{-1} \dd P \ee c k^{-1}
    \ee'^{-1}$
    \\
    & $\stackrel{(\ref{e:b6},\ref{e:b7},\ref{e:b8})}{=}$ & $\ee' c^{-1} \ee^{-1}
    P^{-1} t P k P^{-1} t^{-1} P k^{-1} \ee c \ee'^{-1}$ \\
    & $\stackrel{\eqref{e:b9}}{=}$ & 1.
  \end{tabular}
\end{proof}

\paragraph{The second step in Borisov's approach:} In order to reduce the
number of relations in $\BB$, Borisov adds two new letters $a$ and $b$ while he
encodes some other letters, namely $s_1, \dots, s_N$, $k$ and $t$ as (suitably
complicated) words in $a$ and $b$. In more detail, if $r$ is the $i$th letter
among $s_1, \dots, s_N$, $k$ and $t$, then this letter is encoded as 
$\mu(r) = a^{-1}b^{-1}ab^{-i}ab^{-1}a^{-1}b^ia^{-1}bab^{-i}aba^{-1}b^i$. Then
the presentation is modified so that each relation is replaced with an encoded
version, except the commutators containing $c$ and one of the encoded letters.
They are replaced with $ac = ca$ and $bc = cb$; see the list of relations
below.

\begin{align*}
  d^{\alpha} \mu(s_\beta) &= \mu(s_\beta) d & &\beta \in [N] \tag{BE1}
  \label{e:be1} \\
  e\mu(s_{\beta}) &= \mu(s_{\beta})e^{\alpha} & &\beta \in [N] \tag{BE2}
  \label{e:be2} \\
  d^i\mu(F_i)e^ic &= cd^i\mu(E_i)e^i & &i \in [M] \tag{BE4} \label{e:be4} \\
  d\mu(t) &= \mu(t)d & \tag{BE6} \label{e:be6} \\
  e\mu(k) &= \mu(k)e & \tag{BE8} \label{e:be8} \\
  \mu(P^{-1})\mu(t)\mu(P)\mu(k) &= \mu(k)\mu(P^{-1})\mu(t)\mu(P) &
  \label{e:be9}  \tag{BE9} \\
  ac &= ca & \tag{BE10} \label{e:be10} \\
  bc &= cb & \tag{BE11} \label{e:be11}
\end{align*}

It follows from~\cite[Theorem~2]{borisov69} that a word
$\mu(Q^{-1})\mu(t)\mu(Q)\mu(k)(\mu(k)\mu(Q^{-1})\mu(t)\mu(Q))^{-1}$ is trivial
if and only if $Q$ can be obtained from $P$ in the underlying Thue system.
(Intuitively, $\mu(s_\beta)$, $\mu(k)$ and $\mu(t)$ behave as $s_\beta$, $k$
and $t$ in $\BB$ except that~\eqref{e:be10} and~\eqref{e:be11} together ensure
that these encoded letters commute with $c$.)

We intend to use Theorem~\ref{t:ar_sequence}. For this reason we would like to
get a presentation which becomes trivial after adding a small number of
relations. Unfortunately, Borisov's encoding is not good for this purpose as
the sum of the exponents of $a$ and $b$ in encoded letters equals $0$. On the
other hand, Borisov's proof heavily depends on this particular encoding.

\paragraph{Our modification:} We will modify the encoding so that we make the
occurrences of $a$ and $b$ in the encoded letters slightly imbalanced. The
cost is that we have to provide a new proof that the encoding has desired
properties. Here we use ideas from small cancellation theory. 

Now we describe our final presentation. This time, we will already use the Thue
system of Matiyasevich~\cite{matiyasevich67}, described at the beginning of
Section~\ref{s:ar_sequence}. It has two letters $s_1, s_2$ and three rewriting
rules $F_1 \leftrightarrow
E_1$, $F_2 \leftrightarrow E_2$ and $F_3 \leftrightarrow E_3$. We know that
$F_1 = s_1s_1s_2s_1s_2$, $F_2 = s_1s_1s_2s_2$ and $E_1 = E_2 = s_2s_1s_1$. (We
do not need to specify $F_3$ and $E_3$.)

For encoding, we need to encode letters $s_1, s_2, k$ and $t$. We will again
use $\mu$ for this encoding because we will not need Borisov's original
encoding anymore. We set
\begin{align}
  \mu(s_1) &=
  a^{10}b^{10}a^{-11}b^{-11}a^{-12}b^{-12}a^{13}b^{13}a^{14}b^{14}a^{-15}b^{-15}a^{-16}b^{-16}a^{18}b^{18}
  \label{e:enc_s1} \\
  \mu(s_2) &=  
  a^{20}b^{20}a^{-21}b^{-21}a^{-22}b^{-22}a^{23}b^{23}a^{24}b^{24}a^{-25}b^{-25}
  a^{-26}b^{-26}a^{28}b^{28}\\
  \mu(k) &=  
  a^{30}b^{30}a^{-31}b^{-31}a^{-32}b^{-32}a^{33}b^{33}a^{34}b^{34}a^{-35}b^{-35}
  a^{-36}b^{-36}a^{38}b^{38}\\
  \mu(t) &=  
  a^{40}b^{40}a^{-41}b^{-41}a^{-42}b^{-42}a^{43}b^{43}a^{44}b^{44}a^{-45}b^{-45}a^{-46}b^{-46}a^{48}b^{48}.
  \label{e:enc_t}
\end{align}
This choice is made so that it will satisfy so called metric small cancellation
property $C'(1/6)$ which we will explain in Section~\ref{s:new_encoding}. 
For the moment, we do not need this property. We have also set up the exponents in
such a way that the sum of the exponents of each of $a$ and $b$ is $1$.

Our desired presentation is a presentation $\RR$ with letters $a, b, c, d, e$
and the following 12 relations. (Here $\alpha > 5$.)

\begin{align*}
  d^{\alpha} \mu(s_\beta) &= \mu(s_\beta) d & &\beta \in \{1,2\} \tag{R1}
  \label{e:r1} \\
  e\mu(s_{\beta}) &= \mu(s_{\beta})e^{\alpha} & &\beta \in \{1,2\} \tag{R2}
  \label{e:r2} \\
  d\mu(F_1)ec &= cd^2\mu(E_1)e & \tag{R3.1} \label{e:r31} \\
  d^3\mu(F_2)e^3c &= cd^3\mu(E_2)e^3 & \tag{R3.2} \label{e:r32} \\
  d^4\mu(F_3)e^4c &= cd^4\mu(E_3)e^5& \tag{R3.3} \label{e:r33} \\
  d\mu(t) &= \mu(t)d & \tag{R4} \label{e:r4} \\
  e\mu(k) &= \mu(k)e & \tag{R5} \label{e:r5} \\
  \mu(P^{-1})\mu(t)\mu(P)\mu(k) &= \mu(k)\mu(P^{-1})\mu(t)\mu(P) &
  \label{e:r6}  \tag{R6} \\
  ac &= ca & \tag{R7} \label{e:r7} \\
  bc &= cb & \tag{R8} \label{e:r8}
\end{align*}

We need to know that it is easy to trivialize this presentation.

\begin{lemma}
\label{l:trivialize_R}
  Let $\RR'$ be a presentation obtained from $\RR$ by adding relations $a = 1$
  and $c = 1$. Then $\RR'$ presents a trivial group.
\end{lemma}
\begin{proof}
  Using $c = 1$, we deduce $\mu(F_2) = \mu(E_2)$ from~\eqref{e:r32}. This means
  \[
	  \mu(s_1) \mu(s_1) \mu(s_2) \mu(s_2) = \mu(s_2)\mu(s_1) \mu(s_1).
  \]
  Using $a = 1$, $\mu(s_1)$ as well as $\mu(s_2)$ simplify to $b$. Thus we get
  $b^4 = b^3$ in $\RR'$; that is, $b = 1$.

  Using $a = b= c = 1$, we get $d = 1$ from~\eqref{e:r31} and $e=1$
  from~\eqref{e:r33}.
\end{proof}

We also need to know an analogy of Borisov's result for this presentation

\begin{theorem}
\label{t:R_word_unsolvable}
  Let $Q$ be an admissible word in alphabet $\{s_1, s_2\}$ (no inverses
  allowed). Then 
  \[\mu(Q^{-1})\mu(t)\mu(Q)\mu(k)(\mu(k)\mu(Q^{-1})\mu(t)\mu(Q))^{-1}\]
  is trivial in $\RR$
  if and only if $Q$ can be obtained from $P$ in the underlying
  (Matiyasevich's) Thue system.
\end{theorem}

The proof of this theorem is postponed to Section~\ref{s:new_encoding}.
However, we are ready to deduce Theorem~\ref{t:def9} from Theorem~\ref{t:R_word_unsolvable}.

\begin{proof}[Proof of~Theorem~\ref{t:def9} modulo Theorem~\ref{t:R_word_unsolvable}]
Because the word problem for Matiyasevich's Thue system is undecidable,
  Theorem~\ref{t:R_word_unsolvable} implies that the word problem is
  undecidable for $\RR$. Then Theorem~\ref{t:def9} is an immediate corollary of
  Theorem~\ref{t:ar_sequence} and Lemma~\ref{l:trivialize_R}.
\end{proof}

\section{Modified Borisov's encoding}
\label{s:new_encoding}
The purpose of this section is to prove
Theorem~\ref{t:R_word_unsolvable}. (In this section, in contrast to our
earlier notations, we will denote words with lower case letters while we keep
the capital letters for other objects.) In the following definitions (small
cancellation condition, van Kampen diagrams, etc.), we in general
follow~\cite{lyndon-schupp77}. However, we adjust the terminology a little bit
so that it follows the current standards in graph theory. (For example, a
\emph{path} in~\cite{lyndon-schupp77} allows repetitions of vertices or edges.
But the standard notion in graph theory is a \emph{walk} for this object.)

\paragraph{Small cancellation condition.}
Let $\PP = \langle X | R\rangle$ be a group presentation. A word is
\emph{reduced} if it does not contain a pair of symbols $x$ and $x^{-1}$ next
to each other and it is \emph{cyclically reduced} if each cyclic permutation of
the word is reduced. We will assume that each word in $R$ is cyclically
reduced. (Cyclic permutations of words as well as reducing pairs $x$ and
$x^{-1}$ yields a presentation of the same group.)
The \emph{symmetric closure} of the set $R$ is the set of relations $R^*$ obtained by taking all
cyclic permutations of elements of $R$ and their inverses. We remark that
$\langle X | R\rangle$ and $\langle X | R^*\rangle$ present the same group.
A word $u$ is called a \emph{piece} (with respect to $R$) if there
are distinct $r_1, r_2 \in R^*$ such that $u$ is a maximal common initial segment of
these two words. In addition, by a \emph{subpiece} we mean a subword of a
piece.

For a positive integer $p$ the set of relations $R$ as above satisfies the \emph{small cancellation
condition} $C(p)$ if for every $r \in R^*$ whenever $r$ can be written as a
concatenation $r = u_1\cdots u_s$
where $u_1, \dots, u_s$ are pieces, then $s \geq p$. We also introduce a
slightly modified condition $C_{\sub}(p)$ if the same condition is true also
for subpieces; that is, whenever $r$ can be written as a
concatenation $r = u_1\cdots u_s$
where $u_1, \dots, u_s$ are subpieces, then $s \geq p$.
For $\lambda \in (0,1)$ the set of relations $R$ satisfies the 
\emph{metric small cancellation condition} $C'(\lambda)$ if for every $r \in
R^*$ whenever $r$ can be written as $ux$ where $u$ is a piece, then $|u| <
\lambda|r|$ where $|\cdot|$ stands for the length of a word. We will work only
with conditions $C_{\sub}(6)$ and $C'(1/6)$ and it is easy to check that $C'(1/6)$
implies $C_{\sub}(6)$ (which implies $C(6)$). (Of course, the definitions above make sense for an
arbitrary set of words instead of $R$ as $R$ is just a set of words.)

Given a cyclically reduced word $w$, we say that $w$ is \emph{aperiodic} if eachnontrivial cyclic permutation of $w$ is different from $w$. In other words, $w$ cannot be written as $u^k$ where $u$ is another reduced word and $k \geq 2$.

The following theorem is analogous to Theorem~2 in~\cite{borisov69} and is the
main technical result for the proof of Theorem~\ref{t:R_word_unsolvable}. We
however point out that our proof is very different from the proof
in~\cite{borisov69}. 

\begin{theorem}
  \label{t:encoding}
  Let 
  \[\PP = \langle x_1, \dots, x_p, c, y_1, \dots, y_q| R, y_1 c = c y_1, \dots,
  y_q c = c y_q \rangle\] 
  be a presentation where $p, q$ are non-negative integers and $R$ is an
  unspecified set of relations. 
 
  Consider an encoding of letters of an alphabet $\Sigma = \{x_1, \dots, x_p, c,
  y_1, \dots, y_q\}$ using words in an alphabet $\Sigma' = \{x_1, \dots, x_p, c,
	a, b\}$ given by $\mu(x_i) := x_i$, $\mu(c) := c$, $\mu(y_j) := w_j$ where each $w_j$  is some aperiodic cyclically reduced word using only $a$ and $b$ (and their inverses). We also extend this
  encoding to the inverses (e.g. $\mu(y_j^{-1}) = w_j^{-1}$) and to the words on
  letters of $\Sigma$ and their inverses (e.g. $\mu(x_iy_j^{-1}) =
  \mu(x_i)\mu(y_j^{-1}) = x_iw_j^{-1}$). Assume that the set of (cyclically
  reduced) words $W := \{w_1, \dots,
  w_q\}$ satisfies the metric small cancellation condition $C'(1/6)$; in
  particular, it also satisfies $C_{\sub}(6)$. Assume also
  that if we add
  the relations $\{x_1, \dots, x_p, y_1, \dots, y_q\}$ to $\PP$, then we obtain a
  free group generated by $c$.
  (In other
  words, in every relator of $R$, the sum of the exponents of $c$ is equal to
  $0$.)

  Consider the presentation
  \[\PP' = \langle x_1, \dots, x_p, c, a, b| \mu(R), ac = ca, bc = cb\rangle,\] 
  where $\mu(R)$ stands for $\{\mu(r) \colon r \in R\}$. Then a word $w$ is
  trivial in $\PP$ if and only if $\mu(w)$ is trivial in $\PP'$. 
\end{theorem}

Before we prove Theorem~\ref{t:encoding}, we show how it implies
Theorem~\ref{t:R_word_unsolvable}.

\begin{proof}[Proof of Theorem~\ref{t:R_word_unsolvable} modulo
  Theorem~\ref{t:encoding} and Theorem~\ref{t:B'}]

Let us set $\PP = \BB'$ where $\BB'$ is the presentation from
  Theorem~\ref{t:B'}. We consider $\BB'$ specifically for Matiyasevich's
  Thue system and with $g(1) = h(1) = h'(1) = 1$, $g'(1) = 2$, $g(2) = h(2) =
  g'(2) = h'(2) = 3$, $g(3) = h(3) = g'(3) = 4$ and $h'(3) = 5$. We also
  consider an encoding $\mu$ given by
	formulas~\eqref{e:enc_s1}--\eqref{e:enc_t}. It is trivial to check that these $\mu(s_1)$, $\mu(s_2)$, $\mu(k)$ and $\mu(t)$ are cyclically reduced and aperiodic.
	It is easy to check that the
  set of words $\{\mu(s_1), \mu(s_2), \mu(k), \mu(t)\} =\{w_1, w_2, w_3, w_4\}$ satisfies the metric small
  cancellation condition $C'(6)$. Indeed, let $r \in \{w_1, w_2, w_3, w_4\}^*$
  where $r$ is a cyclic shift of $w_i$. If $r = ux$ where $u$ is a piece, then $u$ is
  of form $a^{j_1}b^{j_2}$ or $b^{j_1}a^{j_2}$ for integers $j_1$, $j_2$
  with $|j_1|, |j_2| \leq 8 + 10i$. Thus $|u|\leq 16 + 20i$. 
  On the other hand, $|r| = |w_i| = 58 + 160i$. Thus $|u| < |r|/6$ as required.
  We also easily verify the condition
  that if we add relations $\{s_1, s_2, d, e, k, t\}$ to $\BB'$, then we get a
  free group generated by $c$. (If we rewrite the relations $u_1 = u_2$ 
  containing $c$ as $u_1(u_2)^{-1} = 1$, then the sum of the exponents of $c$
  equals $0$.)

  This means that we have verified the conditions of Theorem~\ref{t:encoding}.
  The resulting presentation is exactly $\PP' = \RR$. Therefore the fact that 
  \[\mu(Q^{-1})\mu(t)\mu(Q)\mu(k)(\mu(k)\mu(Q^{-1})\mu(t)\mu(Q))^{-1}\]
  is trivial in $\RR$
  if and only if $Q$ can be obtained from $P$ in the underlying Thue system
  follows from Theorem~\ref{t:encoding} and Theorem~\ref{t:B'}.
\end{proof}

\paragraph{Van Kampen diagrams.}

Given a group presentation $\PP = \langle X | R\rangle$, assume that each word
in $R$ is cyclically reduced. A van Kampen diagram $\DD$ over $\PP$ is a finite
$2$-dimensional CW-complex embedded in $\R^2$ with the following properties:
\begin{enumerate}[(i)]
  \item $\DD$ is connected and simply connected.
  \item Each \emph{edge} of $\DD$ is oriented and it is labeled with a letter
    $x \in X$ or its inverse. We consider the labelling of the edge equivalent,
    if we swap the direction of the edge while replacing $x$ for $x^{-1}$ or
    vice versa.
  \item For each $2$-cell, there is a \emph{boundary walk} of the cell. This is
    a closed walk in the $1$-skeleton of $\DD$ along the boundary of the
    cell. (In~\cite{lyndon-schupp77} this is called boundary cycle. But we
    prefer the less specific term walk as repetitions of vertices or even edges are
    possible.)
    If we choose
    a starting vertex and pass this boundary walk in arbitrary direction, we
    read on the edges a word in $R^*$ assuming that we replace a symbol 
    with its inverse if the edge is oriented in opposite direction. The word we
    read is called the \emph{boundary label} of this $2$-cell. It is unique up to
    an inverse and cyclic permutation.
\end{enumerate}

A \emph{boundary label} of a van Kampen diagram $\DD$ is a word read
on the boundary walk of $\DD$ starting at some vertex, choosing some
orientation (and again taking the inverse if the edge is oriented oppositely).
The boundary label is unique up to cyclic permutations and taking the inverse.


\emph{An important convention:} The orientation of edges in (ii) should be
understood as fully flexible. We allow temporarily swapping the orientation of
any edge if we swap its label from $x$ to $x^{-1}$. This will typically occur
when reading the boundary label according to (iii), we will orient the edges in
the same direction around the boundary walk. (Even if the boundary walk
contains a same edge twice passing through it in opposite directions, we can
think of the orientation of the edge according to the walk when passing through
the edge for the first time and then swap the orientation of the edge in the
second pass.)

The following result is known as van Kampen's lemma; see the discussion
in~\cite[Ch. V.1]{lyndon-schupp77}:
\begin{lemma}
\label{l:van_Kampen}
  \begin{enumerate}[(i)]
    \item
      Let $\DD$ be a van Kampen diagram over a presentation $\PP$ and $w$ a
      boundary label of $\DD$. Then $w$ is a trivial word in $\PP$.
    \item 
      Let $w$ be a trivial word in a presentation $\PP = \langle X| R\rangle$
      and assume that $w$ is cyclically reduced. Then there is a van
      Kampen diagram over $\PP$ with boundary label $w$.  
  \end{enumerate}
\end{lemma}

We will be using Lemma~\ref{l:van_Kampen} repeatedly in the proof of
Theorem~\ref{t:encoding} without explicitly
mentioning it. A diagram obtained from (ii) will be called a \emph{diagram
witnessing that $w$ is trivial in $\PP$}.

In the proof of Theorem~\ref{t:encoding}, we will also need the following
variant of Greendlinger's lemma:

\begin{lemma}[Special case of Thm~4.4 in~\cite{lyndon-schupp77}]
\label{l:greendlinger}
Let $\PP = \langle X|R\rangle$ be a group presentation where each word in $R$
  is cyclically reduced. Assume that $R$ satisfies the metric small
  cancellation property $C'(1/6)$. Let $x$ be a reduced word nontrivial in the
    free group on $X$ but trivial in $\PP$. Then $x$ contains a subword $s$ of
    some $r \in R^*$ with $|s| > |r|/2$. 
\end{lemma}

\begin{proof}[Proof of the `only if' part of Theorem~\ref{t:encoding}]
  The `only if' part of the statement is easy: We first observe that the
  relations $ac=ca$ and $bc = cb$ in $\PP'$ imply $\mu(y_j)c = c\mu(y_j)$ for
  every $j \in [q]$ as $\mu(y_j) = w_j$ is a word using only $a$ and $b$ (and
  their inverses). Thus we can add the relations $\mu(y_j)c = c\mu(y_j)$ to
  $\PP'$ without changing the group it presents. Now we assume that $w$ is
  trivial in $\PP$. Thus there is a van Kampen diagram witnessing that $w$ is
  trivial in $\PP$. We subdivide each symbol $y_j$ (or the inverse) in this
  diagram and replace it with $\mu(y_j)$ (or the inverse) while keeping the
	orientation. We obtain a van Kampen diagram witnessing that $\mu(w)$ is trivial in
  $\PP'$ (some cells may use the relations $\mu(y_j)c = c\mu(y_j)$ on their
  boundaries).
\end{proof}

\paragraph{Number of edges in plane graphs.}
In the proof of Theorem~\ref{t:encoding}, we also need the following
bound for the number of edges in a plane graph, possibly with loops and
multiple edges.

\begin{lemma}[{\cite[Lemma~10.2]{matousek-sedgwick-tancer-wagner18}}]
\label{l:mstw}
  Let $G$ be a plane graph (i.e. a graph embedded in the plane) with $n > 2$
  vertices, possibly with loops and multiple edges. Let us also assume that no
  two parallel edges (connecting the same two vertices) and no two parallel
  loops (attached to the same vertex) are isotopic by an isotopy fixing the
  endpoints and avoiding the other vertices. Let us also assume that the
  interior and the exterior of every loop contains a vertex. Then the number of
  edges of $G$ is at most $3n -6$.
\end{lemma}

We remark that Lemma~10.2 in~\cite{matousek-sedgwick-tancer-wagner18} is stated
for embeddings into the sphere $S^2$ but this is not really a difference as any
plane embedding can be easily considered as an embedding in $S^2$ (and vice
versa).

\paragraph{High-level overview of the proof of the `if' part of
Theorem~\ref{t:encoding}.} Because the proof of the `if' part' is technical and
relatively long, we first aim to sketch some of the ideas used in the proof.
This overview may be useful for understanding the full proof but it is not
strictly necessary. We also point out that not all important ideas can be
explained here because some of them require preliminaries done in the full
proof.

The aim of the `if' part of Theorem~\ref{t:encoding} is the following: we
assume that $\mu(w)$ is trivial in $\PP'$ and we want to deduce that $w$ is
trivial in $\PP$. We consider the van Kampen diagram $\DD'$ 
witnessing that $\mu(w)$ is a trivial word in $\PP'$. We would like to build a
van Kampen diagram $\DD$ witnessing that $w$ is trivial in $\PP$.

The boundary labels of $2$-cells of $\DD'$ are cyclic shifts and/or inverses of
$\mu(r)$ where $r$ is a relation in $\PP$ or of $c^{-1}a^{-1}ca$ or
$c^{-1}b^{-1}cb$. In the latter two cases, the corresponding cells are called
the \emph{commutator cells}. If we are lucky, the commutator cells can be
grouped together in the diagram to form larger cells labeled with
$c^{-1}w_j^{-1}cw_j$ where $j \in [q]$ up to an inverse or a cyclic shift. Note
that $w_j = \mu(y_j)$. If in addition, the larger cells and the non-commutator
cells meet nicely, we can replace each path $\mu(y_j)$ with $y_j$ and each
$\mu(r)$ with $r$ thereby obtaining the required diagram $\DD$. (See also
Figure~\ref{f:replace_commutators} used later in the full proof.) 

The difficulty in the proof is that we are not guaranteed to be lucky enough
that all the aforementioned steps work well. In general, the main idea is to
perform some modifications of $\DD'$ so that we are getting closer and closer
to the lucky case. In the full proof, we rule out some specific unwanted
subdiagrams/subconfigurations directly and then we assume that $\DD'$ is
minimal in a certain sense.\footnote{This minimality criterion is however somehow
subtle. For example, we may need to increase the number of cells in order to
get a diagram closer to minimal.} This further allows us to rule out some
unwanted subconfigurations either directly or via some surgery on the diagram.
An example of such surgery is given in Figure~\ref{f:matching_fully} (modifying
the left diagram to the right one). 

We use the (metric) small cancellation property in the proof twice. First we
use it in the form of Lemma~\ref{l:greendlinger} early in the proof to rule out a
few unwanted configurations. A more substantial use of the small cancellation
property appears later in the minimization procedure. If the cells do not fit
well together (as sketched in the `lucky case') then we are able to build a
certain contact representation of disks where the degrees (numbers of touching
points) correspond to the number of (sub)pieces in the small cancellation property.
On the other hand, such a contact representation cannot exist due to the
properties of planar graphs. 

This finishes the overview and now we provide the
full proof.

\begin{proof}[Proof of the `if' part of Theorem~\ref{t:encoding}]
  For the `if' part of the statement, we assume that $\mu(w)$ is trivial in
  $\PP'$ and we want to deduce that $w$ is trivial in $\PP$. First we need to
  check that certain words are nontrivial in $\PP'$. 
  
\begin{claim}
  \label{c:nontrivial}
  Let $k$ be an integer. Then the words
  $ac^{k}$, $bc^{k}$ are nontrivial in $\PP'$. If, in addition, $k \neq 0$, then $c^k$ is
  nontrivial 
  in $\PP'$.
\end{claim}
  
  \begin{proof}
    First we show that $c^k$ is nontrivial in $\PP'$: For each letter $s \in
    \Sigma$ different from $c$ we add a relation $\mu(s)$ to the relations of
    $\PP'$ and we also add relation $a$, $b$, obtaining a presentation $\PP''$.
    By comparing the relations of $\PP$ and $\PP'$, we obtain that $\PP''$
    presents the same group as adding relations $x_1, \dots, x_p, y_1, \dots,
    y_q$ to $\PP$. Therefore $\PP''$ is a free group generated by $c$. It means
    that $c^k$ is nontrivial in $\PP''$ which implies that it is nontrivial in
    $\PP'$ as well.

    Now, for contradiction, assume that $ac^k$ is trivial in $\PP'$. (The proof is analogous
    for $bc^k$.) For each letter $s \in \Sigma$ (including $c$) we add a
    relation $\mu(s)$ to the relations of $\PP'$ obtaining a presentation
    $\PP'''$. As $ac^k$ is trivial in $\PP'$, we get that $a$ is trivial in $\PP'''$.
    By checking the relations of $\PP'''$, we see that we can remove the
    (trivial) letters $x_1, \dots, x_p$ and $c$ and that every other relation
    is implied by relations $\mu(y_1), \dots, \mu(y_q)$, that is, the relations
    $w_1, \dots, w_q$. This means that
    after removing the trivial letters $\PP'''$ is
    equivalently presented as $\langle a, b| W\rangle$. (Recall that $W
    = \{w_1, \dots, w_q\}$.) By
    Lemma~\ref{l:greendlinger}, using $C'(1/6)$, if $a$ is trivial in $\langle a, b| W\rangle$, then a
    subword of $a$ is a subword of some $u \in W^*$ of length more than the
    half of the length of $u$. However, each word $u \in W^*$ has length at
    least two (actually more) which easily follows from the metric small
    cancellation condition $C'(1/6)$ of $W$. (Actually $C(6)$ is
    sufficient.) This contradicts that $a$ is trivial in
    $\PP'''$.
  \end{proof}

  Let $\DD'$ be a
  van Kampen diagram witnessing that $\mu(w)$ is a trivial word in $\PP'$.

  Given a $2$-cell $\sigma$ in $\DD'$ we say that is an \emph{$R$-cell}, if its
  boundary label is of a form $\mu(r)$ for $r \in R^*$ (up to a cyclic
  permutation or taking an inverse). The only remaining cases are that the
  boundary label of $\sigma$ is either $aca^{-1}c^{-1}$ or $bcb^{-1}c^{-1}$ (up
  to a cyclic permutation or taking an inverse). In these cases, we say that
  $\sigma$ is a \emph{commutator cell}.

 Given a $2$-cell $\sigma$ in $\DD'$, we say that $\sigma$ is \emph{regular} if
  the closure of $\sigma$ is homeomorphic to a disk.

\begin{claim}
\label{c:regular}
  Every commutator cell in $\DD'$ is regular.
\end{claim}

\begin{proof}
Assume, for contradiction that there is a commutator cell $\sigma$ which is not
regular. This
    means that the image of the closure of $\sigma$ is not simply connected and
    there is a (closed) disk $D$ such that the boundary of $D$ is formed by
    some proper subword
    of some cyclic permutation and possibly inverse of one of $aca^{-1}c^{-1}$ or
    $bcb^{-1}c^{-1}$. Up to a cyclic permutation and/or inverse,
    $u$ is one of the words $a$, $b$, $c$, $ac$, $ac^{-1}$, $bc$, $bc^{-1}$ or some
    cyclically non-reduced word such as $caa^{-1}$. Each of these words is
    nontrivial due to Claim~\ref{c:nontrivial}. On the other hand, $\DD'$
    restricted to $D$ (as $\DD'$ is simply connected) 
    proves that the aforementioned subword is trivial, which is the required
    contradiction.
\end{proof}

By an \emph{incidence pair} we mean a pair $(\sigma, e)$ where 
\begin{itemize}
 \item  $e$ is an edge labeled $a$, $b$, $a^{-1}$ or $b^{-1}$; and
 \item $\sigma$ is an $R$-cell such that
  $e$ appears on the boundary walk of $\sigma$. 
\end{itemize}
		We also allow $\sigma$ to be
  the outer face of the diagram (by the outer face we mean the complement of
  the diagram in $\R^2$). If $e$ appears on the boundary walk of $\sigma$
  twice, then we have two distinct incidence pairs distinguished by whether
  $\sigma$ appears on the right-hand side or left-hand side. But we do not
  introduce an additional notation in this case.

  Now we define an auxiliary graph $G$ embedded in $\R^2$: The vertices of $G$
  are the midpoints of edges $e$ labelled $a, b, a^{-1}$ or
  $b^{-1}$. Each edge of
  $G$ is represented by a curve within some commutator cell $\sigma$
  connecting the two midpoints of edges with labels $a, b, a^{-1}$ or
  $b^{-1}$. Note that $G$ is embedded and its maximum degree is at most $2$. Note also that the vertices of degree at most $1$ correspond to (the edges of) incidence pairs whereas the vertices of degree 2 do not come from incidence pairs.

  Given two distinct incidence pairs $(\sigma, e)$ and $(\sigma', e')$, we say
  that they are \emph{twins} if the midpoints of $e$ and $e'$ are in the same
  component of $G$; see Figure~\ref{f:twins}. Note that it may either happen that $e = e'$ and the
  midpoint is an isolated vertex of $G$, or $e \neq e'$ and the midpoints of
  $e$ and $e'$ are of degree $1$ in $G$. Note also that for each incidence pair
  there is a uniquely defined twin. We also remark that the edges of twins are
  always labelled with a same letter (if the orientations are suitably chosen).
    
    \begin{figure}
      \begin{center}
	\includegraphics[page=3]{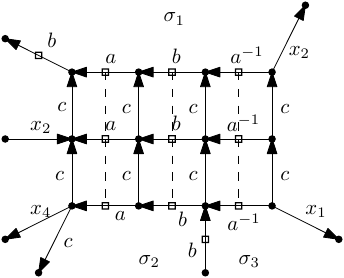}
      \end{center}
      \caption{An example of a part of a diagram showing the graph $G$ and twins.
      The vertices of $G$ are marked by small boxes, the edges of $G$ are
      dashed. For example, the pairs $(\sigma_1, e_1)$ and $(\sigma_2,e_2)$ are
      twins where $e_1$ is the edge labelled $a$ in $\sigma_1$ and $e_2$ is the
      edge labelled $a$ in $\sigma_2$. Also the pairs $(\sigma_2, e_3)$ and
      $(\sigma_3, e_3)$ are twins where $e_3$ is the joint edge of $\sigma_2$
      and $\sigma_3$.}
      \label{f:twins}
    \end{figure}

Now, given an incidence pair $\pp = (\sigma, e)$, recall that the boundary
label of
$\sigma$ is of form $\mu(u)$ where $u$ is a word over $\Sigma \cup \Sigma^{-1}$
where $\Sigma^{-1}$ stands for the set of inverses of letters in $\Sigma$. In more detail,
$u$ might be either $w$ (if $\sigma$ is the outer face), or some $r \in R$
otherwise. Given that $\mu(x_i) = x_i$, $\mu(c) = c$ and $\mu(y_j) =
w_j$, we get that $\mu(u)$ is a concatenation of letters $x_i, x^{-1}_i, c,
c^{-1}$ and words $w_j$ and $w_j^{-1}$ for $i \in [p]$, $j \in [q]$. We do not
even claim that there is a unique way how to write the boundary word of
$\sigma$ as such a concatenation (up to changing the orientation of
$\sigma$).\footnote{With some effort it actually follows from the assumptions
of the theorem that this is unique for a fixed orientation of~$\sigma$. But it
does not help much in the proof of the theorem, thus we rather prefer not
assume this.} We only fix one such way for each $\sigma$ and use it
consistently for all incidence pairs corresponding to $\sigma$. Recall that
$e$ is labelled $a, b, a^{-1}$ or $b^{-1}$. This means that $e$ is inside some
word $w_j$ or $w_j^{-1}$ in the concatenation above. (The decision between $w_j$ or $w_j^{-1}$ depends only on the orientation of $\sigma$.)
Let us denote this word $w(\pp)$
(assuming that we specify the orientation of $\sigma$). (Note that, in
principle, $w(\pp)$ may also be the whole boundary label of
$\sigma$.\footnote{However, this occurs only in a somewhat artificial case when
some $y_j$ (or its inverse) appears in $R$. In such case, $y_j$ can be removed
from the generators and relations without affecting the group presented by
$\PP$.})  

Also the position of $e$ within the word $w(\pp)$ is determined, let
us denote by $i(\pp)$ this position. That is, the label of $e$ is the
$i(\pp)$th letter of $w(\pp)$.

Given twins $\pp = (\sigma, e)$ and $\pp' = (\sigma', e')$, we say that $\pp$
and $\pp'$ are \emph{matched} if the following holds: Once we pick opposite
orientations of $\sigma$ and $\sigma'$, then $w(\pp) = w(\pp')$ and $i(\pp) =
i(\pp')$. In addition, we say that $\pp = (\sigma, e)$ and $\pp' = (\sigma',
e')$ are \emph{fully matched} if they are matched and if the following
condition holds: For every $j$ let $e_j$ or $e'_j$ be the $j$th edge of
$w(\pp) = w(\pp')$ when read along the boundary walk of $\sigma$ or
$\sigma'$ (so that $e = e_{i(\pp)}$ and $e' = e'_{i(\pp)}$). Then we require that $(\sigma, e_j)$ and $(\sigma', e'_j)$ are twins
for every $j$ up to the length of $w(\pp)$.

From now on we assume (without loss of generality) that $\DD'$ satisfies the
following two additional properties.
\begin{enumerate}[(P1)]
  \item Among the diagrams witnessing that $\mu(w)$ is a trivial word, $\DD'$
    has the minimum number of twins which are matched but not fully matched.
  \item Among the diagrams satisfying the previous property, $\DD'$ has the
    minimum number of cells.
\end{enumerate}

\begin{claim}
  \label{c:no_c-cycle}
  There is no closed simple curve in the 1-skeleton of $\DD'$ which would be
  labelled only with letters $c$.
\end{claim}

\begin{proof}
  For contradiction such a curve $\gamma$ exists. Let us pick an orientation of
  $\gamma$ and let $u$ be the word that we read along $\gamma$ with respect to
  a given orientation. This word reduces to $c^k$ where $k$ is an integer. If
  $k \neq 0$, we get a contradiction with Claim~\ref{c:nontrivial} as $\gamma$
  bounds a diagram with a boundary label reducing to $c^k$. 

  It remains to consider the case $k=0$. In such case we adjust orientations of
  edges of $\gamma$ so that each label is $c$ (and not $c^{-1}$). We have a
  same number of edges directed to the left as the number of edges directed to
  the right. We will obtain a new diagram $\DD''$ in the following way.

  We pick a pair of neighboring edges pointing out in opposite directions and
  we identify them by a contraction of the disk bounded by $\gamma$ (understood
  as a contraction of $\R^2$; the result is again homeomorphic to $\R^2$); see
  Figure~\ref{f:contract_gamma}. We
  repeat this step until the disk bounded by $\gamma$ is contracted to a tree.

  \begin{figure}
    \begin{center}
      \includegraphics[page=9]{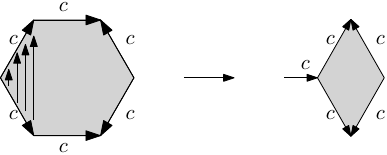}
      \caption{A contraction of the disk bounded by $\gamma$.}
      \label{f:contract_gamma}
    \end{center}
  \end{figure}

  We want to get a contradiction by showing that $\DD'$ was not minimized with
  respect to the properties (P1) and (P2).
  We know that $\DD''$ has the same boundary label as $\DD'$. We also know that
  $\DD''$ has strictly less cells than $\DD'$. 
  
  It remains to check that the
  number of twins in $\DD''$ which are matched but not fully matched is less
  than or equal to the number of those in $\DD'$. We observe that the graph $G$ does not pass through
  $\gamma$. Therefore the contractions we have performed did not affect matched and
  fully matched pairs outside $\gamma$. On the other hand, we might have
  removed some matched pairs but not fully matched pairs inside $\gamma$.
  Therefore $\DD''$ is better with respect to the properties (P1) and (P2)
  which is the required contradiction.
\end{proof}

Let $u$ be a word using only letters in $\{a,b,a^{-1}, b^{-1}\}$ and let $k >
0$ be an integer. We define a \emph{grid diagram} $\GG(u,k)$ (using only commutator
cells) in the following way:

Let $h$ be the length of $w$. We consider the integer $\Z^2$ grid. The set of
vertices of $\GG(u,k)$ is $\Z^2 \cap ([h+1]\times [k+1])$. For every $(i,j) \in
[h] \times [k+1]$ we have a `horizontal' edge directed from $(i,j)$ to $(i+1,j)$
labelled with the $i$th letter of $u$. For every $(i,j) \in [h+1] \times k$ we
have a vertical edge labelled with $c$. The $2$-cells of $\GG(u,k)$  are the
squares bounded by the edges; see Figure~\ref{f:grid_diagram}.

\begin{figure}
  \begin{center}
     \includegraphics[page=15]{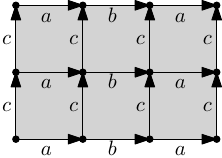}
     \caption{A grid diagram $\GG(aba,2)$.}
     \label{f:grid_diagram}
  \end{center}
\end{figure}

\begin{claim}
\label{c:1-grid}
Assume that $\sigma_1, \dots, \sigma_k$ are distinct commutator cells such that for
  every $i \in [k-1]$ the cells $\sigma_i$ and $\sigma_{i+1}$ share an edge
  labelled with a label in $\{a,b,a^{-1},b^{-1}\}$. Then the subdiagram formed
  by the closure of $\sigma_1 \cup \cdots \cup \sigma_k$ is isomorphic to
  $\GG(\ell,k)$ for some $\ell \in \{a,b,a^{-1},b^{-1}\}$. (In the isomorphism,
  we also allow a reflection of the diagram.)
  
\end{claim}

  
  \begin{figure}
    \begin{center}
      \includegraphics[page=11]{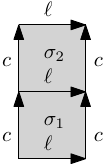}
      \caption{Labels of edges of $\sigma_1$ and $\sigma_2$.}
      \label{f:two_commutators}
    \end{center}
  \end{figure}

\begin{proof}
  The case $k=1$ is the contents of Claim~\ref{c:regular}. We will spend the
  most effort with case $k=2$ which is essentially a straightforward case analysis ruling out all other
  options, using Claims~\ref{c:nontrivial}, \ref{c:regular}
  and~\ref{c:no_c-cycle}. Then, for $k \geq 3$, the claim will easily follow by
  induction.

 First we, therefore, assume $k=2$. Let $P$ be a maximal path shared by boundaries of $\sigma_1$ and $\sigma_2$
  containing an edge $e$ labelled $\ell$. Because we know that each commutator cell
  is regular (Claim~\ref{c:regular}), we know that the edges preceding $P$ and
  the edges following $P$ along each of
  the boundaries are not shared. Let $\sigma$ be an open disk obtained by
  taking the union of $\sigma_1, \sigma_2$ and the interior of $P$. We can
  consider $\sigma$ as a cell in some other diagram obtained by removing the
  edges of $P$ from $\DD'$. The boundary label of $\sigma$ with respect to
  some orientation is then some word that contains at most two pairs of
  edges labelled $c$ and $c^{-1}$ and at most one pair of edges labelled $\ell$
  and $\ell^{-1}$.

  First we observe that the closure of $\sigma$ (which is the same as the
  closure of $\sigma_1 \cup \sigma_2$) is a disk. If this is not the case, then
  the boundary walk of $\sigma$ decomposes in at least two cycles such that the
  words read along these cycles are some subwords $u$ of the boundary label of
  $\sigma$. Due to Claim~\ref{c:no_c-cycle} we know that such a subword contains
  either $\ell$ or $\ell^{-1}$. However, as we have at most one pair of edges
  labelled $\ell$ and $\ell^{-1}$, and at least two cycles, 
  we have exactly one occurrence of $\ell$ and $\ell^{-1}$ together in each of
  these cycles. Due to simple connectedness $u$ bounds some diagram, thus $u$
  is trivial in $\DD'$. But this contradicts Claim~\ref{c:nontrivial}.

  Now we know that the closure of $\sigma$ is a disk. We say that $\sigma_1$
  and $\sigma_2$ is a reducing pair if we pick opposite orientations on
  $\sigma_1$ and $\sigma_2$ then we read the same word along the boundary of
  $\sigma_1$ and $\sigma_2$ when starting with $e$. We will show that
  $\sigma_1$ and $\sigma_2$ cannot be a reducing pair.

  Indeed, if $\sigma_1$ and $\sigma_2$ is a reducing pair, we orient the edges
  on the boundary of $\sigma$ according to the orientations of $\sigma_1$ and
  $\sigma_2$ and we perform a contraction of $\sigma$ identifying the labels
  which is possible since $\sigma_1$ and $\sigma_2$ is a reduction pair
  obtaining a diagram $\DD''$; see Figure~\ref{f:contract_reducing} for an
  example when $P$ has length 1.
  
  \begin{figure}
    \begin{center}
      \includegraphics[page=10]{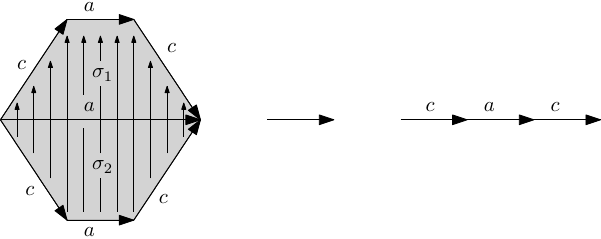}
      \caption{A contraction of a reducing pair.}
      \label{f:contract_reducing}
    \end{center}
  \end{figure}

  The diagram $\DD''$ has the same boundary word as $\DD'$ but it has less
  cells. The counts for property (P1) are the same in $\DD''$ and $\DD'$
  because by considering possible cases, the effect of contraction of a
  reducing pair on $G$ is either a contraction of two edges (if $P$ does not
  have two edges labelled $\ell$ or $\ell^{-1}$) or removing a cycle on two
  edges from $G$ (if $P$ has two edges edges labelled $\ell$ or $\ell^{-1}$).
  None of these changes affects matched or fully matched pairs of twins. Thus
  we get a contradiction that $\DD'$ is a minimal graph satisfying (P1) and
  (P2).  

  Thus, we get that $\sigma_1$ and $\sigma_2$ is not a reducing pair. By
  checking possible overlaps of the commutator cells, we get that the length of
  $P$ is 1. Let us orient the edges on the boundary $\sigma_1$ and $\sigma_2$
  so that we use only the labels $\ell$ and $c$ (but not $\ell^{-1}$ or
  $c^{-1}$). Then the fact that $\sigma_1$ and $\sigma_2$ is not a reducing
  pair translates so that the labelling of the edges on the boundary of
  $\sigma_1$ and $\sigma_2$ is as in Figure~\ref{f:two_commutators} up to
  reverting all orientations of edges labelled $\ell$ and/or reverting all
  orientations of edges labelled $c$. In other words, the subdiagram formed by
  the closure of $\sigma$ is isomorphic to $\GG(\ell, 2)$. (Reverting
  orientations of these collections of edges corresponds to reflections. They
  may possibly swap the position of $\sigma_1$ and $\sigma_2$.)

  To finish the proof, we now assume that $k \geq 3$ and that the claim is true
  for $k-1$. By induction, we know that the subdiagrams formed by closures of
  $\sigma_1 \cup \cdots \cup \sigma_{k-1}$ and $\sigma_2 \cup \cdots \cup
  \sigma_k$ are isomorphic to $\GG(\ell, k-1)$. Therefore, the only way how not
  to arrive at a diagram isomorphic with $\GG(\ell,k)$ is if some vertex of
  $\sigma_1$ not in $\sigma_2$ is identified with some vertex $\sigma_k$ not in
  $\sigma_{k-1}$. This however creates a cycle either with $k$ labels $c$ or
  $k$ labels $c$ and one label $\ell$ or $\ell^{-1}$. This is a contradiction
  with Claim~\ref{c:nontrivial} (taking the subdiagram bounded by the cycle).
\end{proof}

\begin{claim}
\label{c:fully_matched}
  Every matched pair of twins in $\DD'$ is fully matched.
\end{claim}

\begin{proof}
  Assuming that $\DD'$ contains a pair of twins which is matched but not fully
  matched, we will modify the diagram so that we keep the boundary word and we
  decrease the number of matched pairs which are not fully matched. This will
  contradict property (P1).

  Let $\pp
  = (\sigma, e)$ and $\pp' = (\sigma', e')$ be a matched pair of twins in
  $\DD'$ which is
  not fully matched. Let us fix opposite orientations on $\sigma$ and
  $\sigma'$. Let us write $w(\pp) = w(\pp')$ as $\ell_1 \ell_2 \cdots \ell_h$
  where each $\ell_j$ is $a$, $b$, $a^{-1}$ or $b^{-1}$. Let $i := i(\pp) =
  i(\pp')$. Then we get that $e$ as well as $e'$ is labelled with $\ell_i$. 

   Let $e_j$ or $e'_j$ be the $j$th edge of $w(\pp) = w(\pp')$ 
   when read along the boundary walk of $\sigma$ or $\sigma'$. We get that
   $e_j$ as well as $e'_j$ is labelled with $\ell_j$. Let $[i_-, i_+]$ be a
   maximum interval containing $i$ such that for every $j \in [i_-, i_+]$ we
   get that $(\sigma, e_j)$ and $(\sigma', e'_j)$ are twins. This interval may
   contain only one element $i$ but we know that it is not $[h]$ because $\pp$
   and $\pp'$ are not fully matched.

   We will assume that $i_- > 2$ and we will modify the diagram so that we
   extend the maximum shared interval to $[i_- - 1, i_+]$ while keeping all
   other fully matched pairs and keeping the number of incidence pairs.
   Analogous procedure can be also used if $i_+ < h$. Therefore, after a finite
   number of steps, we extend the shared interval to $[h] = [1,h]$ thereby obtaining
   new fully matched pairs.

  Let us denote some of the vertices of our diagram $v_1, \dots, v_{h+1}$ so that
  the walk $v_1, e_1, v_2, e_2, \dots, e_h, v_{h+1}$ is the walk along which we
  read the word $w(\pp)$ in $\sigma$. (See Figure~\ref{f:matching_fully},
  left.) Analogously, we denote $v'_1, \dots
  v'_{h+1}$ along $\sigma'$. Let $P$ be the path in graph $G$ 
  connecting the midpoints of the edges $e_{i_-}$ and $e'_{i_-}$. This path may possibly consist of a single
  vertex. Then, using Claim~\ref{c:1-grid} on commutator cells along $P$, there
  is a unique path $Q$ (possibly trivial) from $v_{i_-}$ to $v'_{i_-}$
  in $\DD'$ such that each edge of $Q$ is parallel with some edge of $P$
  within some commutator cell. Following the direction form $v_{i_-}$ to $v'_{i_-}$, this path $Q$ uses only the edges of commutator cells
  all labelled $c$ or all labelled $c^{-1}$.
  By $Q'$ we denote the walk obtained from $Q$ by
  prepending the vertex $v_{i_- -1}$ and the edge $e_{i_- -1}$ and by appending
  the edge $e'_{i_- -1}$ and the vertex $v'_{i_- -1}$. The labels of edges of
  $Q'$ are $c$, $c^{-1}$, $\ell_{i_- -1}$ or $\ell_{i_- -1}^{-1}$.

    \begin{figure}
      \begin{center}
	\includegraphics[page=4]{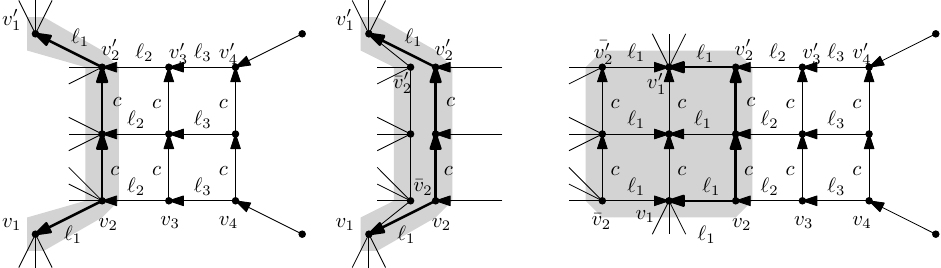}
	\caption{Left: An example of a diagram $\DD'$ with a pair of twins
	which is matched but not fully matched. Here $[i_-, i_+] = [2,3]$, $Q'$
	is denoted by thick edges and $N$ is in gray. Middle: First
	modifications of $\DD'$; $N$ is drawn slightly thicker for an easier
	visualisation. Right: The second modification, adding the commutator
	cells, up to a homeomorphism for easier visualisation.}
	\label{f:matching_fully}
      \end{center}
    \end{figure}

 We will distinguish whether $Q'$ is a path or not.

 We start with the case that $Q'$ is not a path. 
 Because $Q$ is a path while $Q'$ is not a path, this
 means that either one of the vertices $v_{i^- - 1}$ or $v'_{i^- - 1}$ is
 identified with a vertex of $Q$ or $v_{i^- - 1} = v'_{i^- - 1}$.

 The former case yields to a contradiction with Claim~\ref{c:nontrivial} (used on
 the diagram bounded by the cycle obtained from the identification of the vertices).
 Therefore $v_{i^- - 1} = v'_{i^- - 1}$ and they are distinct from any vertex
 of $Q$. In addition $Q$ has to consist of a single vertex, otherwise we get a
 contradiction with Claim~\ref{c:nontrivial} again (all edges along $Q$ have a
 same label $c$ or $c^{-1}$ when $Q$ is oriented from $v_{i^-}$ to $v'_{i^-}$). That is, $v_{i^-} =
 v'_{i^-}$. However, $e_{i^- -1} \neq e'_{i^- -1}$ because we assume that 
 $(\sigma, e_{i^- -1})$ and $(\sigma', e'_{i^- -1})$ are not twins; see
 Figure~\ref{f:bigon}, left. In this case, we remove the bigon between the
 edges $e_{i^- - 1}$ and $e'_{i^- - 1}$ and replace it with an edge $\bar
 e_{i^- - 1}$ (oriented according to the orientations of $\sigma$ and
 $\sigma'$ and keeping the label); see Figure~\ref{f:bigon}, right. This modification preserves the
 boundary labels of cells (those that haven't been removed). In addition,
 $(\sigma, \bar e_{i^- - 1})$ and $(\sigma', \bar e_{i^- - 1})$ is now pair of
 twins. This means that any pair of twins which was fully matched before the
 modification is still fully matched after the modification and we have
 extended the interval between $\sigma$ and $\sigma'$ of matched pairs. After
 repeating this procedure now with the shared interval $[i_- -1, i^+]$ and
 using analogous procedure for $i_+ < h$, we will get all twins along the
 shared interval fully matched which is the required contradiction. (Of course, sometimes, we may fall into the
 case that $Q'$ is a path, which remains to be explained.)

\begin{figure}
  \begin{center}
    \includegraphics[page=12]{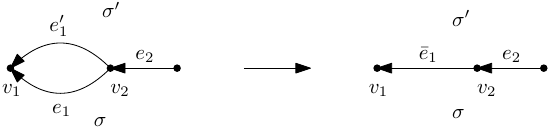}
    \caption{Resolving the case when $Q'$ is not a path. (In the drawing we
    assume that $i^- = 2$.)}
    \label{f:bigon}
  \end{center}
\end{figure}

 Now we assume that $Q'$ is a path. We take a regular neighborhood $N$ of $Q'$ and we modify $\DD'$ inside
  this regular neighborhood. Because $Q'$ is a path we know that $N$ is a disk.
Let us think of $Q'$ as oriented in the plane so that the edges $e_{i_-}$ and
$e'_{i_-}$ enter $Q'$ from the right. We will take a copy of $Q'$ called $\bar
Q'$ connecting $v_1$ and $v'_1$ but shifting all other vertices of $\bar Q'$ a
little bit to the left (that is the vertices which originally belonged to $Q$.
For a vertex $z$ or $Q$ let $\bar z$ be the corresponding vertex on $\bar Q$.
Whenever some edge enters $z$ in $\DD'$ from a left, we modify the diagram so that it now
enters $\bar z$ from the left (and we keep the edge outside $N$); see
Figure~\ref{f:matching_fully}, middle.

Now we orient and label $\bar Q'$ in the same way as $Q$. Then $Q'$ and $\bar
Q'$ bound a cell with a boundary label (word) of a form $\ell_{i_- -
1}\ell_{i_- - 1}^{-1}u\ell_{i_- -1}\ell_{i_- - 1}^{-1}u^{-1}$ where $u$ is a
word using only $c$ or $c^{-1}$. This cell can be easily filled with commutator
cells unless $u$ is trivial; see Figure~\ref{f:matching_fully}, right. If $u$
is trivial, we instead merge the vertices $v_1$ and $v'_1$ and the edges
connecting these vertices to $v_2 = v'_2$ as well as the edges
connecting them to $\bar v_2 = \bar v'_2$. This contracts the cell bounded by
$Q'$ and $\bar Q'$.  


It is easy to check that after these modifications we do not change boundary
labels of cells except possibly adding new commutator cells. In addition each
pair of fully matched twins remains fully matched (such twins were not touched
during the modification). After repeating this procedure now with the shared
interval $[i_- - 1, i_+]$ and using an analogous procedure if $i_+ < h$, we get
all twins along the shared interval are fully matched. We have increased
the number of fully matched pairs as required which proves the claim.
\end{proof}


\begin{claim}
\label{c:fully_matched}
  Every pair of twins in $\DD'$ is matched thereby fully matched.
\end{claim}

\begin{proof}
  For purpose of this proof we consider pairs $(\sigma, W)$ where $\sigma$ is
  some cell (or outer face) and
  $W$ is a subwalk  of the boundary walk of $\sigma$ which corresponds to some
  word $w_j = \mu(y_j)$. We say that $(\sigma, W)$ unmatched if for some edge
  $e$ of $W$, the incidence pair $(\sigma, e)$ is not matched with its twin.
  Because we already know that matched and fully matched is equivalent, this
  condition is equivalent with asking that for every edge $e$ of $W$, the
  incidence pair $(\sigma, e)$ is not matched with its twin. For contradiction,
  we assume that there is at least one unmatched pair $(\sigma, W)$.

  For each unmatched pair $\ww = (\sigma, W)$ we create (an image of) a disk $D_{\ww}$ in
  $\R^2$ in the following way. We push the boundary walk $W$ a little bit inside
  $\sigma$ while removing duplicated vertices. This way we obtain a path $\bar
  W$. (In a special case when $W$ corresponds to $w_j$ which is the
  whole boundary label of $\sigma$, $W$ starts and ends in a same
  vertex. Then we disconnect $W$ at the start/end-vertex when pushing. Thus
  $\bar W$ is still a path.)
  We thicken this path to a disk $D'_{\ww}$ inside $\sigma$ which is not yet
  $D_{\ww}$. We also require that distinct $\ww$ yield disjoint disks which can be achieved if we also
  shorten $\bar W$ a little bit before thickening to $D'_{\ww}$. See
  Figure~\ref{f:pushing_disks} for an example.
    
    \begin{figure}
      \begin{center}
	\includegraphics[page=5]{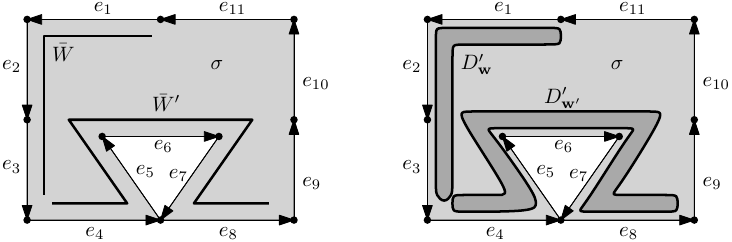}
	\caption{Paths $\bar W$ and $\bar W'$ and disks $D'_{\ww}$ and
	$D'_{\ww'}$ for unmatched pairs $\ww = (\sigma, W)$ and $\ww' =
	(\sigma, W')$ where $W$ is a walk given by edges $e_1$, $e_2$, $e_3$
	and $W'$ is given by edges $e_4, \dots, e_8$.}
	\label{f:pushing_disks}
      \end{center}
    \end{figure}

  Now let us consider some incidence pair $\pp = (\sigma, e)$ such that $e \in W$.
  Let $\pp' = (\sigma', e')$ be a twin of $\pp$. Note that $\pp'$ corresponds to some unmatched
  pair $\ww' = (\sigma', W')$. Now we pull finger moves from disks $D'_{\ww}$
  and $D'_{\ww'}$ along the graph $G$ (the path connecting the midpoints of $e$
  and $e'$) so that the disks $D'_{\ww}$ and $D'_{\ww'}$ touch after this
  modification; see Figure~\ref{f:pulling_fingers} for an example. This we do
  for every pair of twins. This way we obtain the $D_{\ww}$. (Note that
  if $\ww = \ww'$, then pulling fingers causes that $D_{\ww}$ is not a closed
  disk anymore; however, the interior of $D_{\ww}$ is still an open disk.)

    \begin{figure}
      \begin{center}
	\includegraphics[page=6]{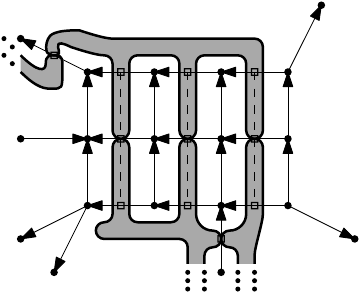}
	\caption{Pulling fingers from the disks $D'_{\ww}$ along paths in graph
	$G$. The diagram corresponds to the diagram from Figure~\ref{f:twins}.}
	\label{f:pulling_fingers}
      \end{center}
    \end{figure}

  At this moment, we have a \emph{contact representation} of disks $D_\ww$
  (with possible self-identifications of the boundary). We
  consider the plane dual graph $H$ of this contact representation. The
  vertices of $H$ are points $v_\ww$ in $D_\ww$ and edges are curves inside the
  disks such that for each touching point there is exactly one edge passing
  through the touching point and connecting the vertices of the disks touching
  at this touching point. The edges do not cross.

We note that $H$ may contain loops and multiple edges.
  Now we argue that each loop in $H$ contains a vertex of $H$ both in its
  interior and its exterior. (By the interior, we mean the interior of the disk
  bounded by this loop and by the exterior we mean the plane minus this disk.)
  Indeed, for contradiction, let $\lambda$ be a loop in $H$ which does not contain a
  vertex of $H$ in its interior. In addition, we can assume that $\lambda$ is an
  innermost such loop. Let $v_\ww$ be the endpoint of the loop (i. e., the loop
  appears inside $D_\ww$). Locally, close  to $v_\ww$, the two arcs of
  $\lambda$ have to be consecutive because any other edge entering the interior
  of $\lambda$ would be either a loop inside $\lambda$ or a non-loop edge with other
  endpoint inside $\lambda$. Then the fact that the arcs of $\lambda$ are consecutive
  corresponds to the fact that the word $w_j$ contains (possibly after a cyclic
  shift) two consecutive letters $\ell\ell^{-1}$ where $w_j$ is the word read
  along the walk $W$, if $\ww = (\sigma, W)$. See Figure~\ref{f:no_empty_loop}
  for a nontrivial example when a cyclic shift is required. This contradicts the assumption
  that the words $w_j$ are cyclically reduced. An analogous argument, using the
  outermost loop, reveals that each loop contains a vertex in its exterior.

    \begin{figure}
      \begin{center}
	\includegraphics[page=18]{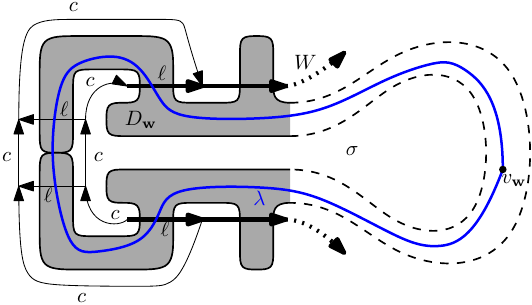}
	\caption{A nontrivial example with an empty loop. In this example,
	$w_j$ would be of a form $\ell u \ell^{-1}$ when read in a clockwise direction
	along $\sigma$ where $u$ is another word. (The edges of the walk $W$ are thick.)}
	\label{f:no_empty_loop}
      \end{center}
    \end{figure}

Given two edges (possibly loops) of $H$, we say that they are \emph{equivalent}, if
they are isotopic via an isotopy fixing the endpoints and avoiding other
vertices of $H$. We build a graph $H'$ from $H$ by keeping the vertices and 
picking exactly one representative from each maximal collection of of
equivalent edges and loops. Now $H'$ satisfies the assumptions of
Lemma~\ref{l:mstw}, thus we conclude that $H'$ has at most $3n - 6$ edges
where $n$ is the number of vertices of $H$. 

Finally we aim to show that the degree of each vertex in $H'$ is at least $6$
which will yield the desired contradiction.

Consider a maximal collection $\varepsilon_1, \dots, \varepsilon_k$ of equivalent
edges (possibly loops). 
Assume that the endpoints of each of these edges
are $v_{\ww}$ and $v_{\ww'}$ with $\ww = (\sigma, W)$ and $\ww' = (\sigma',
W')$. (We have $\ww = \ww'$ if they are loops.) Then
$\varepsilon_1, \dots, \varepsilon_k$ emanate from $v_\ww$ consecutively without loss of generality
in this order in the clockwise direction. (Here we already use that
each loop contains a vertex in the interior and in the exterior, otherwise an
empty loop could appear in between.) Subsequently, $\varepsilon_1, \dots,
\varepsilon_k$ emanate
from $v_{\ww'}$ consecutively in the counterclockwise direction. (We consider
the other ends of these edges, if they are loops.) According to the
construction of $H$, each $\varepsilon_i$ corresponds to a touching point of
$D_\ww$ and $D_{\ww'}$. The touching points appear consecutively exactly in the same order in
which $\varepsilon_i$ emanate from $v_\ww$ and $v_{\ww'}$ respectively.
The touching points of $D_{\ww}$ and $D_{\ww}$ are in a bijection with pairs of
  twins $\pp = (\sigma, e)$ and $\pp' = (\sigma',e')$ which are not matched.
  Furthermore $e$ belongs to some walk $W$ where $\ww = (\sigma, W)$ is an
  unmatched pair as well as $e'$ belongs to some walk $W'$ where $\ww' =
  (\sigma', W')$ is an unmatched pair. A collection of consecutive touching
  points along a boundary of $D_{\ww}$ as well as along the boundary $D_{\ww'}$
  corresponds to the fact that words $w(\pp)$ and $w(\pp')$ possibly
  after a cyclic shift\footnote{A cyclic shift is necessary if the consecutive
  touching points are the first and the last touching point of $D_\ww$ or
  $D_{\ww}$ when the first and the last is considered in the order induced by
  $W$ or $W'$ respectively.} share a subpiece
  (formed by the labels of letters of the edges in pairs of twins corresponding
  to the touching points). Note that this is indeed a proper subpiece as either
  $w(\pp) \neq w(\pp')$ or $i(\pp) \neq i(\pp')$ because $\pp$ and $\pp'$ are
  not matched. (Here we also use our assumption that the words $w_j$ are
  aperiodic. It is used in the case that $w(\pp) = w(\pp')$ and $i(\pp) \neq
  i(\pp')$.) Due to the small cancellation property for subpieces $C_{\sub}(6)$, we get that
  $w(\pp)$ meets this way at least six subpieces and thus
  there
  are at least six maximal collections of equivalent loops or multiple edges
  from each vertex of $H$. That is, the degree of each vertex in $H'$ is at
  least six.

Finally, given that the degree of each vertex in $H'$ is at least $6$, the
handshaking lemma (in graph theory) implies the number of edges of $H'$ is at least $3n$ which is a
contradiction with our earlier claim that the number of edges is at most
$3n-6$.
\end{proof}

Next we observe, by contradiction, that $G$ does not contain any cycle (that is $G$ is a disjoint
union of paths). Indeed, this is an immediate corollary of
Claim~\ref{c:1-grid} used on commutator cells $\sigma_1, \dots, \sigma_k$ where
$\sigma_i$ is the commutator cell corresponding to $i$th edge of the cycle
(once we fix the first edge and an orientation). Then $\sigma_1$ and $\sigma_k$
cannot close the cycle.

Now let $\kappa$ be a commutator cell. It corresponds to an edge in $G$ (passing
through $\kappa$) and consequently (because $G$ has no cycles) 
to a pair of twins $\pp = (\sigma, e)$ and $\pp' = (\sigma', e')$ such that the
path in $G$ connecting the midpoints of $e$ and $e'$ passes through $\kappa$. Let
$W(\pp)$ be the walk on boundary of $\sigma$ corresponding to the word $w(\pp)$
and $W(\pp')$ be the one on boundary of $\sigma'$ corresponding to $w(\pp')$.
Let $e_j$ be the $j$th edge of $W(\pp)$ and $e'_j$ the $j$th edge of $W(\pp')$,
using the same conventions as when defining a fully matched pair. Due to
Claim~\ref{c:fully_matched}, we know that $w(\pp) = w(\pp')$ (in suitable
orientations) and setting $h$ to be the length of $w(\pp)$ we also know that
$(\sigma,e_j)$ and $(\sigma', e'_j)$ are twins for every $j \in [h]$. Let $v_j$
be the $j$th vertex of $W(\pp)$ and let
$v'_j$ be the $j$th vertex of $W(\pp')$ for $j \in [h+1]$. Finally, for $j \in
[h+1]$ let $Q^-_j = Q^-_j(\kappa)$ be the path from $v_j$ to $v'_j$
using only the edges labelled $c$ or $c^{-1}$ in commutator cells
along the path in $G$ connecting the midpoints of $e_j$ and
$e'_j$. By $Q^+_j = Q^+_j(\kappa)$ we denote the analogous path from $v_{j+1}$
to $v'_{j+1}$. (Here we use Claim~\ref{c:1-grid} to check that such paths
exist.)

\begin{claim}
\label{c:replace_commutators}
  The concatenation of walks $W(\pp)$, $Q^+_{h}$, $W(\pp')$ (in opposite direction)
  and $Q^-_1$ (in opposite direction) forms a simple cycle. The diagram inside
  this cycle is isomorphic with $\GG(w(\pp),k)$ for some $k > 0$.
\end{claim}

\begin{proof}
  Let $j_1 \leq j_2$ be two elements of $[h]$ let $W_{j_1,j_2}(\pp)$ (or
  $W_{j_1,j_2}(\pp')$) denote the subwalk of $W(\pp)$
  (or $W(\pp')$) formed by the
  edges $e_j$ (or $e'_j$) for $j \in [j_1,j_2]$ respectively. Similarly, let
  $w_{j_1,j_2}(\pp)$ be the word formed by the $j$th letters of $w(\pp)$ for $j
  \in [j_1,j_2]$.

  Let $j_0 \in [h]$ be such that $e = e_{j_0}$. (That is, the path in $G$ 
  connecting the midpoints of $e$ and $e'$ intersects $\kappa$.)

  We first show, by induction in $j_2$, that for every $j_2 \in [j_0,h]$ the concatenation of walks
  $W_{j_0,j_2}(\pp)$,
  $Q^+_{j_2}$, $W_{j_0,j_2}(\pp')$ (in opposite direction) and $Q^-_{j_0}$ (in 
  opposite direction) forms a simple cycle and the diagram inside this
  cycle is isomorphic with $\GG(w_{j_0,j_2}(\pp),k)$ for some $k > 0$.

  From
  Claim~\ref{c:1-grid}, used on commutators along the path in $G$ connecting
  the midpoints of $e_{j_0}$ and $e'_{j_0}$ we get the first induction step if
  $j_2 = j_0$.
  
  For the second induction step, we assume that
  the concatenation of walks
  $W_{j_0,j_2-1}(\pp)$,
  $Q^+_{j_2-1}$, $W_{j_0,j_2-1}(\pp')$ (in opposite direction) and $Q^-_{j_0}$ (in
  opposite direction) forms a simple cycle and the diagram inside this
  cycle is isomorphic with $\GG(w_{j_0,j_2-1}(\pp),k)$ for some $k > 0$.
  Because $k > 0$, we in particular get that $v_{j_2} \neq v'_{j_2}$. This
  means that there is at least one commutator cell along the path connecting
  the midpoints of $e_{j_2}$ and $e'_{j_2}$. By Claim~\ref{c:1-grid}, used on
  commutators used along the aforementioned path, we get that 
  the concatenation of walks
  $W_{j_2,j_2}(\pp)$,
  $Q^+_{j_2}$, $W_{j_2,j_2}(\pp')$ (in opposite direction) and $Q^-_{j_2}$ (in
  opposite direction) forms a simple cycle and the diagram inside this
  cycle is isomorphic with $\GG(w_{j_2,j_2}(\pp),k')$ for some $k' > 0$; see
  Figure~\ref{f:merge_grid} according to our current knowledge.

  \begin{figure}
    \begin{center}
      \includegraphics[page=16]{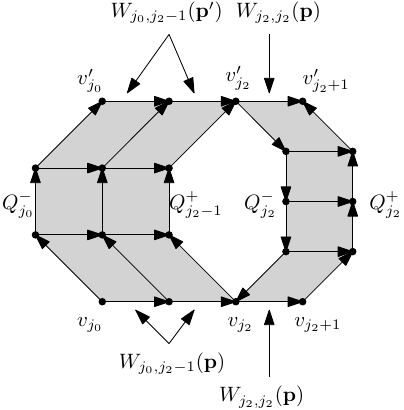}
      \caption{Two concatenated walks during the induction with $k=3$ and
      $k'=4$.}
      \label{f:merge_grid}
    \end{center}
  \end{figure}

  Next we observe that $Q^+_{j_2-1}$ and $Q^-_{j_2}$ have to be identical. If
  this is not the case, then we get cycle with labels only $c$ or $c^{-1}$
  contradicting Claim~\ref{c:no_c-cycle}. However this already means that $k =
  k'$ and the diagrams isomorphic with $\GG(w_{j_0,j_2-1}(\pp),k)$ and
  $\GG(w_{j_2,j_2}(\pp),k')$ in our notation merge to a diagram isomorphic with
  $\GG(w_{j_0,j_2}(\pp),k)$ in our notation. (Note that the concatenation of
  walks $W_{j_0,j_2-1}(\pp)$ and $W_{j_2,j_2}(\pp)$ is just the walk
  $W_{j_0,j_2}(\pp)$ and an analogous claim holds for walks $W'_{j_0,j_2-1}(\pp)$
  etc.)

  If we use $j_2 = h$ in the claim we just proved: The concatenation of walks
  $W_{j_0,h}(\pp)$,
  $Q^+_{h}$, $W_{j_0,h}(\pp')$ (in opposite direction) and $Q^-_{j_0}$ (in
  opposite direction) forms a simple cycle and the diagram inside this
  cycle is isomorphic with $\GG(w_{j_0,h}(\pp),k)$ for some $k > 0$.

  Now, if we use an analogous induction in reverse direction (decreasing $j_0$
  down to $1$), we prove the required claim.
\end{proof}

Now we have acquired all tools to build a diagram $\DD$ witnessing that $w$ is
a trivial word in $\PP$. This will finish the proof of Theorem~\ref{t:encoding}.
In order to get $\DD$ we will be relabelling some edges in $\DD'$ as well as
merging some cells in $\DD'$.

First we take care of commutator cells. Given a commutator cell $\kappa$, let
$W(\kappa)$ be the walk (cycle by Claim~\ref{c:replace_commutators}) obtained
by concatenation of $W(\pp)$, $Q^+_h$,
$W(\pp')$ and $Q_1^{-}$ (last two in opposite directions) as in
Claim~\ref{c:replace_commutators}. In addition, let $\DD'(\kappa)$ be the
subdiagram bounded by $W(\kappa)$. From Claim~\ref{c:replace_commutators} we
know that if $W(\kappa) \neq W(\kappa')$ for some other commutator cell
$\kappa'$, then $W(\kappa)$ and $W(\kappa')$ cannot be nested. Therefore, from
the construction of $W(\kappa)$ it follows that for $\kappa \neq \kappa'$
either $\DD'(\kappa) = \DD'(\kappa')$, or $\DD'(\kappa)$ and $\DD'(\kappa')$ are
disjoint possibly up to the boundary.

Now for any fixed $\kappa$ as above, we know that $\DD'(\kappa)$ is isomorphic
to the grid diagram $\GG(\ww(\pp), k)$ for some $k > 0$, using again the notation
from Claim~\ref{c:replace_commutators}. In addition, $\ww(\pp) = w_j$ for some
$j \in [q]$. We perform the following changes: We merge all edges of $W(\pp)$
(considering $W(\pp)$ oriented so that we read $\ww(\pp)$ according to this
orientation without switching to inverses) into a single directed cell with the
same orientation. On this edge we put label $y_j$. Similarly, for $W(\pp')$ we
merge the edges into a single edge putting the label $y_j$ again. Then we
replace $\DD'(\kappa)$ with the grid diagram $\GG(y_j, k)$ so that we preserve
all orientations (possibly using reflections of the grid diagram); see
Figure~\ref{f:replace_commutators}. Note that these replacements are compatible
with the rest of the diagrams because $W(\pp)$ and $W(\pp')$ are also walks on
boundaries of non-commutator cells. Therefore the inner vertices of $W(\pp)$ or
$W(\pp')$ do not have neighbors outside $\DD'(\kappa)$, thus replacing $W(\pp)$
or $W(\pp')$ with a single edge is possible. (Of course, we also use that
$W(\pp)$ and $W(\pp')$ are paths.)

\begin{figure}
  \begin{center}
    \includegraphics[page=17]{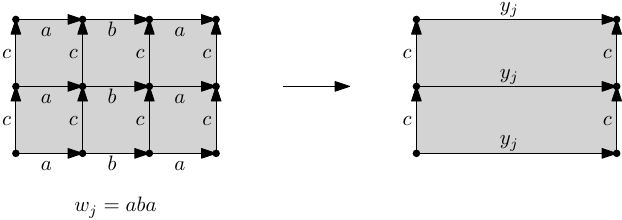}
    \caption{Replacement of commutator cells.}
    \label{f:replace_commutators}
  \end{center}
\end{figure}

We perform these replacements repeatedly as long as there is a commutator cell
$\kappa$ for which $\DD'(\kappa)$ has not been replaced yet.

Now we consider an arbitrary $R$-cell $\sigma$ (or the outer face). If we
consider the boundary walk of $\sigma$, the word we originally read in $\DD'$
was a concatenation of words $w_j$ for $j \in [q]$, or letters $x_i$ for $i \in
[p]$ or the letter $c$, or inverses of the aforementioned words and letters.
Some of $w_j$ might have been already replaced with $y_j$ due to treating
commutator cells. If there is however $w_j$ which haven't been replaced yet, we
consider the corresponding walk $W$ on the boundary of $\sigma$ representing
$w_j$. Due to the fact that every pair of twins is fully matched
(Claim~\ref{c:fully_matched}), this is also
a walk on boundary of a neighboring cell $\sigma'$ (it may happen that $\sigma
= \sigma'$). This means, that $W$ is actually a path. We orient this path in
the direction in which we read $w_j$ and we replace this path with an edge
labelled with same orientation labelled $y_j$.

We perform these replacements as long as there is a $\sigma$ and $w_j$
admitting such a replacement. After all these replacements we get the desired
diagram $\DD$. In diagram $\DD$ all occurrences of $w_j$ have been replaced with
$y_j$, other symbols have been either kept, or some symbols $c$ might have been
removed if they appeared inside some grid diagrams. However, in any case, for
the boundary word we have performed the inverse of the encoding $\ell \to
\mu(\ell)$ for all letters $\ell$ on the boundary, and similarly, for all cells
which we use in $\DD$ we have also performed the inverse of the encoding. This
means that $\DD$ witnesses that $w$ is a trivial word in $\PP$.
\end{proof}

\section*{Acknowledgments} 
I would like to thank David E. Speyer for an answer
to my MathOverflow question and Connor Malin for an useful comment on it. I
would also like to thank Cameron McA. Gordon for pointing out an inaccuracy regarding
the smooth structures in an early version of this manuscript. Finally, I would
like to thank an anonymous referee for reading the very paper carefully and
pointing out a few other inaccuracies (that have been corrected).

\bibliographystyle{alpha}
\bibliography{markovs}

\appendix
\section{Proof of Theorem~\ref{t:B'}}
\label{a:borisov_modified}

In this section, our aim is to prove the `if part' of Theorem~\ref{t:B'}. The proof is
essentially identical to the proof of Theorem~1 in~\cite{borisov69}
(Theorem~\ref{t:borisov} in our notation). When comparing
Theorem~\ref{t:B'} and Theorem~1 in~\cite{borisov69}, let us recall that
the only change that we have performed is that we have replaced presentation
$\BB$ by presentation $\BB'$. In more detail, we have removed the relations 
  
\begin{align*}  
  d^iF_ie^ic &= cd^iE_ie^i & &i \in [M] \tag{B4} 
\end{align*}
and replaced them with the relations
\begin{align}
  d^{g(i)}F_ie^{h(i)}c &= cd^{g'(i)}E_ie^{h'(i)} & &i \in [M] \tag{B4'},
\end{align}
where $g, h, g', h'$ are injective functions and the parameter $\alpha$ (for
$\BB$ or $\BB'$) is strictly larger than any value attained by these function.

For verifying Theorem~\ref{t:B'} one extreme would be to announce that a
careful reading of the proof of Theorem~1 in~\cite{borisov69} gives a proof of
(the `if part' of) Theorem~\ref{t:B'} essentially in verbatim. Another
extreme would be to rewrite/copy-paste the technical proof from~\cite{borisov69}
including auxiliary results in order that the proof can be fully understood
while making necessary adjustments. This would mean rewriting/copy-pasting the
majority of Borisov's paper without a clear benefit because the adjustments are
rather minor. 

We have decided for an intermediate way: We will exhaustively pinpoint all
locations in Borisov's proof where an adjustment is necessary and we check that
the same conclusions can be obtained also in our setting. From this point of
view, the proof of Theorem~\ref{t:B'} is not self-contained but it heavily
uses (an English translation of)~\cite{borisov69}. On the other hand, our
attempt is to make the individual statements and conclusions understandable
without discussing~\cite{borisov69}, if the reader is willing to accept that
those are the only places where some adjustments are necessary.

\begin{proof}[Proof of the `if' part of Theorem~\ref{t:B'}---adjustments
  when compared with a proof of Theorem~1 in~\cite{borisov69}]

  Borisov's proof starts with five assertions (mostly) on regularity of certain
  transient letters. (Undefined here.) The first three assertions remain
  unaffected by our change. The assertions IV. and V. are affected; they
  however have the same proof, thus we explain the modifications only once.

 The key in the proof of IV. and V. is to show that certain words $A_\ell$
  defined as $A_\beta = s_\beta$ for $\beta \in [N]$ and $A_{N+i} = d^i F_i
  e^i$ for $i \in [M]$ are free generators (of a subgroup they generate). (There
  is also an analogous statement for certain $B_\ell$ but this one is
  symmetric.) Borisov proceeds with a proof by contradiction and he considers
  several cases 1), 2), 3) and 4) how a trivial (reduced) word $W$ can be composed from the powers of $A_\ell$. In our setting $A_{N+1}$ is replaced with $d^{g(i)} F_i
  e^{h(i)}$. 

  In case 1), Borisov finds a subword of a form $e^{i_1-i_2}$ which has to be
  expressed as a word in $d$ and $e^{\alpha}$. From this he deduces $i_1=i_2$
  as $|i_1 - i_2| < \alpha$ which gives a reducing pair in $W$. In our setting,
  we get a word $e^{h(i_1) - h(i_2)}$. However, the conclusion is the same
  $|h(i_1) - h(i_2)| < \alpha$ which implies $h(i_1) = h(i_2)$. But this also
  implies $i_1 = i_2$ as $h$ is injective.

  Case 2) yield expressing $e^i$ as a word in $d$ and $e^{\alpha}$ which
  implies $i = 0$ (as $i < \alpha$) which yields a contradiction. In our case,
  we get $e^{h(i)}$ getting $h(i)=0$ and yielding the same contradiction. Case
  3) is analogous to case 2) and case 4) is not affected by our change.

  In the statement of V. (used later on), the expression $V_{N+i} = d^i D_ie^i$
  gets replaced either with $V_{N+i} = d^{g(i)} D_ie^{h(i)}$ or $V_{N+i} =
  d^{g'(i)} D_ie^{h'(i)}$ as the need arises. However, this choice is made
  consistently for all $i$ (either we use $g(i)$ and $h(i)$ for all $i$ or we
  use $g'(i)$ and $h'(i)$ for all $i$).

\medskip

The following arguments are not affected by our change until the proof of
  Lemma~4 in~\cite{borisov69}. Proof of this lemma concludes the proof of the
  theorem.  Following the notation of~\cite{borisov69}, we write $U \bequal V$ if two
  words $U$ and $V$ are identical, letter by letter. Lemma~4 states that if $Y
  \bequal P^{-1}LQR$ where $L$ is an irreducible word on letters $c$ and $d$
  (and their inverses) and $R$ is an irreducible word on letters $c$ and $e$
  (and their inverses), and in addition $Y$ is trivial in a presentation
  (group) $\Gamma_2$
  given by relations~\eqref{e:b1}--\eqref{e:b4}, then $P$ and $Q$ can be
  obtained one from another in the underlying Thue system. The proof of Lemma~4
  in~\cite{borisov69} is given by induction on $m$ where $m$ is the number of
  positive occurrences of $c$ in $Y$.

  The first place in the proof of Lemma~4 in~\cite{borisov69} where our change
  of presentation plays a role is a consideration of `case (2)' in this proof.
  In this case, Borisov obtains a word $X$ which can be written as $X \bequal
  d^fQe^r B^{-\sigma_1}_{\ell_1} \cdots B^{-\sigma_s}_{\ell_s}$ while $X$ is
  trivial in a presentation (group) $\Gamma_3$ given by relations~\eqref{e:b1}
  and~\eqref{e:b2}. The words $B_{\ell_j}$ are either $B_\beta = s_\beta$ for
  $\beta \in [N]$, or $B_{N+i} = d^i E_i e^i$ (these are right hand sides
  of~\eqref{e:b3} and~\eqref{e:b4} after removing $c$ and they are counterparts
  of words $A_\ell$ discussed earlier). In our modification of the
  presentation, we have $B_{N+i} = d^{g'(i)} E_i e^{h'(i)}$.

 Next step is to show that subwords $(d^iE_ie^i)^\varepsilon$ for $\varepsilon
 = \pm1$ do not appear more than once in the word $V \bequal
 B^{\sigma_1}_{\ell_1} \cdots B^{\sigma_s}_{\ell_s}$.\footnote{Based on $X \bequal
  d^fQe^r B^{-\sigma_1}_{\ell_1} \cdots B^{-\sigma_s}_{\ell_s}$ and the
  subsequent discussion we expect that we should rather have $V \bequal
 B^{-\sigma_1}_{\ell_1} \cdots B^{-\sigma_s}_{\ell_s}$ but this is a minor
 detail as changing the signs of $\sigma_j$ does not affect the proof.} In our
 case, we of course want to prove that subwords
 $(d^{g'(i)}E_ie^{h'(i)})^\varepsilon$ do not appear more than once.

The desired contradiction is obtained by a series of computations (these are
the displayed math formulas at the end of page 771 in the English translation
of~\cite{borisov69}). The final step is a deduction that $e^{r_3}d^{-i_1
\alpha^{\partial(U_2)}}e^{-i_2}$ is a word in $d$ and $e^{\alpha}$ where
$\partial(U_2)$ stands for length of a certain word $U_2$ and $i_1 \neq 0$. But
this is a contradiction as $i_2 < \alpha$. In our case, the analogous
computations yield that $e^{r_3}d^{-g'(i_1)
\alpha^{\partial(U_2)}}e^{-h'(i_2)}$ is a word in $d$ and $e^{\alpha}$ with
$g'(i_1) \neq 0$. But this is again a contradiction because $h'(i_2) < \alpha$
as well.

After this, Borisov discusses the cases whether $d^i E_i e^i$ appears exactly
once, or it does not appear in $V$.

In the former case, this allows to express $P^{-1}LQR$ from the statement of
Lemma~4 as
\[
  P^{-1} L' cd^{i\alpha^{\partial(K')}}K'E_i K''
  e^{i\alpha^{\partial(K'')}}c^{-1}R'
\]
where $L'$ is some initial segment of $L$, $Q \bequal K'E_iK''$  and $R'$ is
some terminal segment of $R$. In our case, we get
\[
  P^{-1} L' cd^{g'(i)\alpha^{\partial(K')}}K'E_i K''
  e^{h'(i)\alpha^{\partial(K'')}}c^{-1}R'.
\]
This word is trivial in $\Gamma_2$ due to the assumption that $P^{-1}LQR$ is
trivial in $\Gamma_2$. Now, we reiterate Borisov's computations in this
setting (all equalities are in $\Gamma_2$; we repeatedly use \eqref{e:b1},
\eqref{e:b2}, \eqref{e:b3} and~\eqref{e:b4'}):

\begin{align*}
  P^{-1} L' cd^{g'(i)\alpha^{\partial(K')}}K'E_i K''
  e^{h'(i)\alpha^{\partial(K'')}}c^{-1}R' &= 1 \\
  P^{-1} L' K' cd^{g'(i)}E_i e^{h'(i)} c^{-1} K''
  R' &= 1 \\
  P^{-1} L' K' d^{g(i)}F_i e^{h(i)} c c^{-1} K''
  R' &= 1 \\
  P^{-1} L' d^{g(i)\alpha^{\partial(K')}}K' F_i K'' e^{h(i)\alpha^{\partial(K'')}}  R' &= 1 \\
\end{align*}

The last word can be expressed as $PL_1K'F_iK''R_1 \bequal PL_1Q_1R_1$ where $L_1 = L'
d^{g(i)\alpha^{\partial(K')}}$ and $R_1 = e^{h(i)\alpha^{\partial(K'')}}  R'$.
This allows to use induction as explained in~\cite{borisov69}.

Finally, the latter case when $d^i E_i e^i$ does not appear in $V$ is not
affected by our change of presentation and allows induction in the same way as in~\cite{borisov69}.
\end{proof}

\end{document}